\documentclass[a4paper,reqno]{amsart}

\usepackage{amsmath,amsthm,amssymb,mathrsfs}
\usepackage[alphabetic,nobysame]{amsrefs}
\usepackage{mathtools}
\usepackage{lmodern}
\usepackage{microtype}
\usepackage[T1]{fontenc}
\usepackage{dsfont}
\usepackage[english]{babel}
\usepackage[hidelinks]{hyperref}
\usepackage[textsize=tiny, color=white, linecolor=lightgray]{todonotes}

\usepackage{tikz-cd}

\allowdisplaybreaks

\newtheorem{theorem}{Theorem}[section]
\newtheorem{lemma}[theorem]{Lemma}
\newtheorem{prop}[theorem]{Proposition}
\newtheorem{cor}[theorem]{Corollary}

\theoremstyle{definition}

\newtheorem{remark}[theorem]{Remark}

\numberwithin{equation}{section}

\DeclareMathOperator{\Br}{Br}
\DeclareMathOperator{\Cl}{Cl}
\DeclareMathOperator{\Hr}{H}
\DeclareMathOperator{\Pic}{Pic}
\DeclareMathOperator*{\im}{im}
\DeclareMathOperator{\Hom}{Hom}
\DeclareMathOperator*{\Spec}{Spec}
\DeclareMathOperator{\vol}{vol}

\newcommand{\Cd}{\mathds{C}}
\newcommand{\Zd}{\mathds{Z}}
\newcommand{\Fd}{\mathds{F}}
\newcommand{\Rd}{\mathds{R}}
\newcommand{\Gd}{\mathds{G}}
\newcommand{\Pd}{\mathds{P}}
\newcommand{\Ad}{\mathds{A}}

\newcommand{\Gm}{\Gd_{\mathrm{m}}}

\newcommand{\ind}{\mathds{1}}

\newcommand{\Sc}{\mathcal{S}}
\newcommand{\Cc}{\mathcal{C}}
\newcommand{\Pc}{\mathcal{P}}
\newcommand{\Ic}{\mathcal{I}}
\newcommand{\Lc}{\mathcal{L}}
\newcommand{\Oc}{\mathcal{O}}
\newcommand{\Qc}{\mathcal{Q}}
\newcommand{\Os}{\mathcal{O}}
\newcommand{\Uc}{\mathcal{U}}
\newcommand{\Is}{\mathcal{I}}
\newcommand{\OK}{\of_K}
\newcommand{\IK}{\Is_K}

\newcommand{\Ab}{\mathbf{A}}
\newcommand{\Ib}{\mathbf{I}}
\newcommand\oneb{\mathbf{1}}

\newcommand{\ab}{{\boldsymbol{a}}}
\newcommand{\db}{{\boldsymbol{d}}}
\newcommand{\rb}{{\boldsymbol{r}}}
\newcommand{\ub}{{\boldsymbol{u}}}
\newcommand{\xb}{{\boldsymbol{x}}}
\newcommand{\yb}{{\boldsymbol{y}}}

\newcommand{\xbb}{{\underline{\xb}}}
\newcommand{\ybb}{{\underline{\yb}}}

\newcommand{\Df}{\mathfrak{D}}

\newcommand{\af}{\mathfrak{a}}
\newcommand{\bfr}{\mathfrak{b}}
\newcommand{\cf}{\mathfrak{c}}
\newcommand{\dfr}{\mathfrak{d}}
\newcommand{\efr}{\mathfrak{e}}
\newcommand{\ffr}{\mathfrak{f}}
\newcommand{\Lf}{\mathfrak{L}}
\newcommand{\pf}{\mathfrak{p}}
\newcommand{\qf}{\mathfrak{q}}
\newcommand{\kf}{\mathfrak{k}}
\newcommand{\of}{\mathfrak{o}}
\newcommand{\Xf}{\mathfrak{X}}
\newcommand{\Yf}{\mathfrak{Y}}
\newcommand{\N}{\mathfrak{N}}
\newcommand{\Pf}{\mathfrak{P}}
\newcommand{\rfr}{\mathfrak{r}}
\newcommand\gammab{\boldsymbol{\gamma}}

\newcommand{\cI}{{}_\cfrb I}
\newcommand{\cR}{{}_\cfrb R}
\newcommand{\cYf}{{}_\cfrb\Yf}

\newcommand{\Fs}{\mathcal{F}}
\newcommand{\Ms}{\mathcal{M}}
\newcommand{\Gs}{\mathcal{G}}
\newcommand\Msunder{\underline{\Ms}}
\newcommand\Msover{\overline{\Ms}}
\newcommand\Munder{\underline{M}}
\newcommand\Mover{\overline{M}}
\newcommand\id{\mathrm{id}}

\newcommand\Mod[1]{\ (\mathrm{mod}\ {#1})}
\newcommand\congr[3]{#1 \equiv #2 \Mod{#3}}

\newcommand{\cfb}{\underline{\mathfrak{c}}}

\newcommand{\df}{\mathrm{d}}
\newcommand{\br}{\mathrm{b}}

\newcommand{\sums}[1]{\sum_{\substack{#1}}}
\newcommand{\prods}[1]{\prod_{\substack{#1}}}
\newcommand{\norm}[1]{\lVert #1 \rVert}

\newcommand\cfrb{{\underline{\mathfrak{c}}}}
\newcommand\dfrb{{\underline{\mathfrak{d}}}}
\newcommand\efrb{{\underline{\mathfrak{e}}}}
\newcommand\afrb{{\underline{\mathfrak{a}}}}

\AtBeginDocument{%
   \def\MR#1{}
}

\begin{document}

\title{Equidistribution and the torsor method}

\author[C.~Bernert]{Christian Bernert}

\address{Institute of Science and Technology Austria, Am Campus 1, 3400 Klosterneuburg, Austria}

\email{\href{mailto:christian.bernert@ist.ac.at}{christian.bernert@ist.ac.at}}

\author[U.~Derenthal]{Ulrich Derenthal} 

\address{Institut f\"ur Algebra, Zahlentheorie und Diskrete Mathematik, Leibniz Universit\"at Hannover, Welfengarten 1, 30167 Hannover, Germany}

\email{\href{mailto:derenthal@math.uni-hannover.de}{derenthal@math.uni-hannover.de}}

\author[F.~Wilsch]{Florian Wilsch}

\address{Mathematisches Institut, Georg-August-Universität Göttingen, Bunsenstra\ss{}e 3--5, 37073 Göttingen, Germany}

\email{\href{mailto:florian.wilsch@mathematik.uni-goettingen.de}{florian.wilsch@mathematik.uni-goettingen.de}}

\date{July 16, 2026}


\setcounter{tocdepth}{1}

\begin{abstract}
    We prove equidistribution and Manin's conjecture for rational points outside the lines on smooth split quintic del Pezzo surfaces over number fields with respect to any anticanonical height. The proof is based on a general theorem that is broadly usable to deduce equidistribution when using the torsor method with several equivalent height functions to treat variants of Manin's problem.
\end{abstract}

\maketitle

\tableofcontents

\section{Introduction}

What is the distribution of rational points on algebraic varieties? 

\smallskip

This classical question in number theory can be interpreted in several different ways.
For a variety $X$ over a number field $K$ with infinitely many rational points, a qualitative, topological point of view is to study weak approximation, that is, to ask whether the set $X(K)$ of rational points is dense in the space of adelic points. Even for arithmetically relatively simple classes of varieties, this property can fail to hold, and the goal is thus to find and study criteria for when it holds or fails. For contributions to this problem, see \cites{CTS87,Harari04,Peyre05,Wittenberg23}, for instance.

A standard quantitative approach to this question introduces a height function $H \colon X(K) \to \Rd_{>0}$ measuring the complexity of rational points and studies the asymptotic growth rate of the number of rational points of bounded height. For (almost) Fano varieties, Manin's conjecture~\cite{FMT89} and its refinements and generalizations due to Batyrev, Chambert-Loir, Peyre, Tschinkel \cites{Pey95,BT98,CLT10}, and others are guiding the research on this problem, while several variants of this problem on integral points~\cites{CLT12,Wil24,Santens23}, Campana points~\cites{PSTVA,MR5034466}, or rational points in families~\cites{MR3568035,loughran2024leadingconstantrationalpoints} are now active areas of research as well.

Equidistribution combines these perspectives: where weak approximation guarantees the existence of at least one rational point in every open set of adelic points, equidistribution is about the number of such rational points of bounded height. For a projective variety $X$ over a number field $K$, this means that there is a probability measure $\tau$ on the space $X(\Ab_K)$ of adelic point that quantifies the proportion of rational points that lie in sufficiently well-formed sets. Concretely, one says that a subset $V\subset X(K)$ of rational points is \emph{equidistributed in (a quotient of) $X(\Ab_K)$ with limit measure $\tau$ when ordered by height $H$} if, for every $W$ in (that quotient of) $X(\Ab_K)$ whose boundary is a $\tau$-null set, we have
\begin{equation*}
    \frac{|\{x \in V \cap W : H(x) \le B\}|}{|\{x \in V : H(x) \le B\}|} \to \tau(W),
\end{equation*}
as $B \to \infty$ (see \cite[D\'ef.~3.1]{Pey95}). More formally, this can be regarded as an instance of weak convergence of probability measures. Note that even in the absence of weak approximation, equidistribution in this sense may hold: the limit measure $\tau$ might simply be supported on a strict subset of $X(\Ab_K)$.

\subsection{From counting to equidistribution}

This point of view on rational points in the context of Manin's conjecture goes back to Peyre~\cite{Pey95}. While equidistribution seems significantly more fine-grained than pure counting prima facie, his work contains a key connection between the two problems.
His interpretation of the leading constant in Manin's problem features \emph{Tamagawa measures} that depend on a metric on the anticanonical bundle. One feature of this interpretation is that a modification of the metric is equivalent to putting a greater or lower weight on rational points inside some region of the space of adelic points. As a consequence, an asymptotic formula for all such metrics (of the shape predicted by Peyre) implies an equidistribution theorem for rational points~\cite[Prop.~3.3]{Pey95}. In fact, it is not necessary to deal with all such metrics; a dense subset --- such as that of smooth metrics --- is sufficient.

Attacks on Manin's problem based on harmonic analysis can readily deal with arbitrary smooth metrics, and several equidistribution results have been obtained  based on Peyre's theorem; see \cites{CLT02,CLT12} for (partial) equivariant compactifications of vector groups, for example. In the torsor method for Manin's conjecture, arbitrary smooth adelic metrics are significantly less natural; sharp cutoff functions, turned into explicit polynomial inequalities with the help of a Cox ring, are instead a standard feature. The aim of this article is thus to provide a variant of Peyre's abstract equidistribution theorem that can be easily and broadly applied to works based on the torsor method.

\subsection{An abstract equidistribution theorem for the torsor method}

To this end, let $X$ be a (not necessarily smooth) projective variety over a number field $K$, equipped with a big and semiample line bundle $L$. Let $V\subset X(K)$ be some set of rational points; in the classical setting of Manin's conjecture, this would be the set of rational points that do not lie on an accumulating subset, but in fact, our main theorem is applicable more broadly to sets $V$ of integral or semi-integral points, for instance.

Heights for the torsor method are obtained by passing to a multiple $dL$ of $L$ (with $d \in \Zd_{>0}$) that is base point free and selecting a finite set $\Pc = \{s_0,\dots, s_n\} \subset \Hr^0(X,dL)$ of global sections without common zeros. Pulling back the standard (supremum) metric on $\Oc_{\Pd^n}(1)$ along the morphism $f_{\Pc}\colon X\to \Pd^n$ induced by $\Pc$ yields one on $dL$ and $L$. Call the resulting metrized bundle $L_\Pc$ --- it induces a height function $H_{\Pc}$, which allows us to define the counting function
\begin{equation*}
    N(\Pc,B) = |\{x\in V : H_{\Pc}(x)\le B\}|
\end{equation*}
for the number of points in $V$ of bounded height. If height functions associated with sufficiently many such sets $\Pc$ result in an asymptotic formula whose leading constant is proportional to a measure $\tau$, it turns out that equidistribution follows. In the context of Manin's conjecture for rational points, this measure is the restriction of the Tamagawa measure to the Brauer--Manin set $X(\Ab_K)^{\Br X}$. Our first main result is thus the following equidistribution theorem, generalizing the one of Peyre~\cite[Prop.~3.3]{Pey95} and that by Chambert-Loir and Tschinkel~\cite[Prop.~2.10]{CLT10} as applicable to varieties.

\begin{theorem}\label{thm:equidistribution-theorem}
    Let $X$ be a projective variety over a number field $K$ equipped with a big and semiample line bundle $L$, and let $dL$ be a base point free multiple. Let $V\subset X(K)$ be a subset of its rational points. Let $a\in \Rd_{> 0}$, $b\in \Rd_{\ge 1}$, and $c\in \Rd_{>0}$. Let $S$ be a finite set of places of $K$ containing the archimedean ones. Let $\Pc= \{s_0,\dots, s_n\} \subset \Hr^0(X,dL)$ be a finite set of global sections without common zeroes on $X$ and $f = f_{\Pc}$ the corresponding morphism. Let $\tau_\Pc$ be a finite measure on $X(\Ab_K)$, and for each set $\Pc'$ with the same properties, define a measure $\tau_{\Pc'}$ via
    \begin{equation}\label{eq:tamagawa-relative-heights}
        \df \tau_{\Pc'} = \left(\frac{H_{\Pc}}{H_{\Pc'}}\right)^a\df \tau_{\Pc}.
    \end{equation}
    
    If, for each $\Pc'\subset \Hr^0(X,dL)$ such that
    \begin{enumerate}
        \item $\Pc\subset \Pc'$ and \label{enum:P'-cond-1}
        \item every section in $\Pc'\setminus \Pc$ has the form $uh+u'h'$ with $u,u'\in \of_S$ and $h,h'\in \Pc$,\label{enum:P'-cond-2}
    \end{enumerate}
    the asymptotic expansion
    \begin{equation}\label{eq:asymptotic-hypothesis}
        N(\Pc',B)  = c \tau_{\Pc'}(X(\Ab_K)) B^a (\log B)^{b-1} (1+o_{\Pc'}(1))
    \end{equation}
    holds for $B \to \infty$, then the images of points in $V$ under $f$ are equidistributed in $\prod_{v\in S} f(X)(K_v)$ with respect to the limit measure $f_*\tau_\Pc$ (normalized to a probability measure) when ordered by $H_\Pc$.
\end{theorem}
Note that while an extension of a height function $H_\Pc = \norm{s}_{L_\Pc}^{-1}$ to adelic points depends on the choice of a meromorphic section $s$, the quotient in~\eqref{eq:tamagawa-relative-heights} does not. Moreover, if $f$ identifies points in $V$, they have to be counted with multiplicity in the equidistribution property. 

\begin{cor}\label{cor:adelic-equidistribution}
   Keep all the assumptions in Theorem~\ref{thm:equidistribution-theorem}, but assume that~\eqref{eq:tamagawa-relative-heights} and~\eqref{eq:asymptotic-hypothesis} hold for all
   $\Pc'\subset \Hr^0(X,dL)$ such that
    \begin{enumerate}
        \item $\Pc\subset \Pc'$ and
        \item every section $s\in \Pc'\setminus \Pc$ has the form $uh+u'h'$ with $u,u'\in K$ and $h,h'\in \Pc$.
    \end{enumerate}
    
    Let $H$ be an arbitrary Arakelov height function associated with $L$ and define
    \begin{equation*}
        \df \tau_H = \left(\frac{H_\Pc}{H}\right)^a \df \tau_{\Pc}.
    \end{equation*}
    Then the images of points in $V$ are equidistributed in $f(X)(\Ab_K)$ with respect to the limit measure $f_*\tau_H$ (normalized to a probability measure) when ordered by $H$, and
    \begin{equation*}
        |\{x\in V : H(x)\le B\}| = c \tau_H(X(\Ab_K))B^a (\log B)^{b-1} (1+o_H(1)).
    \end{equation*}
\end{cor}

An alternative route --- and the one we shall take in Section~\ref{sec:dp5} in an example --- to a full equidistribution result is to use Theorem~\ref{thm:equidistribution-theorem} to treat archimedean places and add congruence conditions to extend the theorem to all places (Corollary~\ref{cor:equidistribution-by-counting-mod-q}).

\begin{remark}
    The compatibility property~\eqref{eq:tamagawa-relative-heights} is standard for Tamagawa measures as defined by Peyre~\cite{Pey95} as well as for their variants, such as for integral points~\cite{CLT10}, Campana points~\cites{PSTVA,MR5034466}, and rational points in families~\cites{loughran2024leadingconstantrationalpoints}. In each case, it immediately follows from a definition in the shape
    \begin{equation*}
        \df\tau = \prod_{v\in \Omega_K}\frac{\lambda_v}{\norm{s_v(\phi)}_{aL,v}}\df \phi^{-1}_*\mu,
    \end{equation*}
    where the $\lambda_v$ are some convergence factors independent of the metric, $s_v(\phi)$ is a section depending on a local chart $\phi$ but independent of the metric, and $\mu$ is a Haar measure: in Peyre's classical setting, an $n$-form, and for integral points, the product of the canonical section of (a subset of) the boundary with an $n$-form at finite places or the lift of a form on certain boundary strata over archimedean places, for instance.
\end{remark}

\begin{remark}
    If $X$ is a smooth, projective variety with $\of_S$-model $\Xf$, equipped with a big and semiample line bundle $L$ and such that a multiple $dL$ admits a globally generated model $\Lf$, then the set 
    \begin{equation*}
        \Sc = \left\{
            \Qc\subset \Hr^0(\Xf,\Lf) :
            |\Qc|<\infty
            \text{ and } 
            \Qc \text{ globally generates $\Lf$}
        \right\}
    \end{equation*}
    of sections contains all $\Pc'$ that satisfy conditions~\ref{enum:P'-cond-1} and \ref{enum:P'-cond-2} in Theorem~\ref{thm:equidistribution-theorem}, where $\Pc$ is any fixed system of polynomials drawn from this set. Hence, it suffices to check~\eqref{eq:asymptotic-hypothesis} for $\Qc$ in $\Sc$, and the same is true for variants $\Sc'$ of $\Sc$ that impose additional restrictions such as $K\Qc =\Hr^0(X,dL)$ that keep $\Sc'$ nonempty and stable under $\of_S$-linear combinations.
    
    The approach in Section~\ref{sec:dp5} will take this route. The advantage of the set $\Sc$ is that all its members $\Qc = \{s_0,\dots,s_n\}$ induce height conditions that are trivial at places outside $S$. If $x\colon \Spec \of_\pf \to \Xf$ is a local point, then $x^*\Lf\cong \of_\pf$, and the pull-backs
    $x^*s_0,\dots, x^*s_n\in x^*\Lf$ generate the residue field $\Fd_\pf$,
    so that $\max\{|s_0(x)|_\pf,\dots, |s_n(x)|_\pf\}=1$ (independently of the trivialization $x^*\Lf \cong \of_\pf$). Thus $\lVert s(x)\rVert_\pf =1$ for all sections $s$ such that $s(x)\in \Lf(x)$ is primitive. In other words, that $\Qc$ globally generates $\Lf$ means that $\lVert\cdot \rVert_\pf$ is the model norm and the vector $(s_0(x),\dots, s_n(x))$ is primitive for every $x$.
\end{remark}

\subsection{An application to quintic del Pezzo surfaces}

Theorem~\ref{thm:equidistribution-theorem} can be used out of the box to obtain equidistribution statements over archimedean places as corollaries to counting theorems --- both for rational and integral points --- in the literature~\cites{BD24,BD25integral}, which deal with sufficiently general heights. Recall that a smooth split quintic del Pezzo surface over a field $K$ is obtained by blowing up $\Pd^2_K$ in four $K$-rational points in general position. Up to a projective linear change of variables, any such set of points coincides with
\begin{equation}\label{eq:4_points}
    p_1=(1:0:0),\quad p_2=(0:1:0),\quad p_3=(0:0:1),\quad p_4=(1:1:1).
\end{equation}
In an anticanonical embedding, $X$ contains ten lines, which are precisely its $(-1)$-curves. Let $H$ (resp. $H_D$) be one of the (log) anticanonical height functions and $\tau_H$ (resp. $\tau_D$) be the corresponding Tamagawa measure treated in these works.

\begin{cor}\label{cor:rational}
    Let $X$ be a smooth split quintic del Pezzo surface over a number field $K$. Then the set $U(K)$ of rational points outside the lines is equidistributed in $\prod_{v \in \infty_K} X(K_v)$ with limit measure $\tau_H$ when ordered by $H$.
\end{cor}

\begin{cor}\label{cor:integral}
Let $\Xf$ be the smooth split quintic del Pezzo surface over $\OK$ obtained by blowing up $\Pd^2_{\OK}$ in the four points $p_1,\dots,p_4$ as above,
and let $\Df$ be the strict transform of a line passing through two of the centers of the blow-up. Then the set of integral points on $\Xf\setminus \Df$ that (as rational points) lie outside the lines is equidistributed in $\prod_{v\in \infty_K} X(K_v)$ with limit measure $\tau_{D}$ when ordered by $H_D$.
\end{cor}

The second part of this article is devoted to strengthening Corollary~\ref{cor:rational} to a full equidistribution theorem inside the space of adelic points. In principle, a similar treatment of integral points should be possible.

\begin{theorem}\label{thm:equidistribution_dP5}
    Let $X$ be a smooth split quintic del Pezzo surface over a number field $K$. Then the set $U(K)$ of rational points outside the lines is equidistributed in $X(\Ab_K)$ with limit measure $\tau_H$ when ordered by any anticanonical height $H$. In particular, Manin's conjecture holds for arbitrary anticanonical height functions.
\end{theorem}

Points in $U(K)$ of height at most 500 (with a height function $H$ induced by the set $\{P_\sigma, Q_\sigma : \sigma \in S_3\}$ of cubic forms as in~\cite[p.~2]{BD24}) along barycentric coordinates on $\Pd^2$ are shown in Figure~\ref{fig:dp5}. The six lines arising as strict transforms of lines through two centers of the blow-up are shown in gray, while the remaining four lines are exceptional divisors and contracted along this map.
The points of bounded height are color-coded according to the 15 possible residue disks modulo $2$ that they can lie in: three of these are around points away from the exceptional divisors, while each of the four exceptional divisors admits three more points modulo $2$. 
Theorem~\ref{thm:equidistribution_dP5} then implies that the number of brown rational points in the black triangle is asymptotically $1/180$ of the total number of points (the $2$-adic measure is uniform over the residue disks, yielding a factor $1/15$, and the real measure of the triangle yields a factor $1/12$ by its symmetry; in general, the real contribution can be determined by integrating the real density function over that region).
Note that even though the Tamagawa measure corresponding to this height function is symmetric with respect to reflections along the lines, the barycentric coordinate functions do not follow this symmetry. Hence, the density of this measure with respect to the Lebesgue measure of the chart decreases towards the edges of Figure~\ref{fig:dp5}.

\begin{figure}
    \begin{center}
        \includegraphics[width=\linewidth]{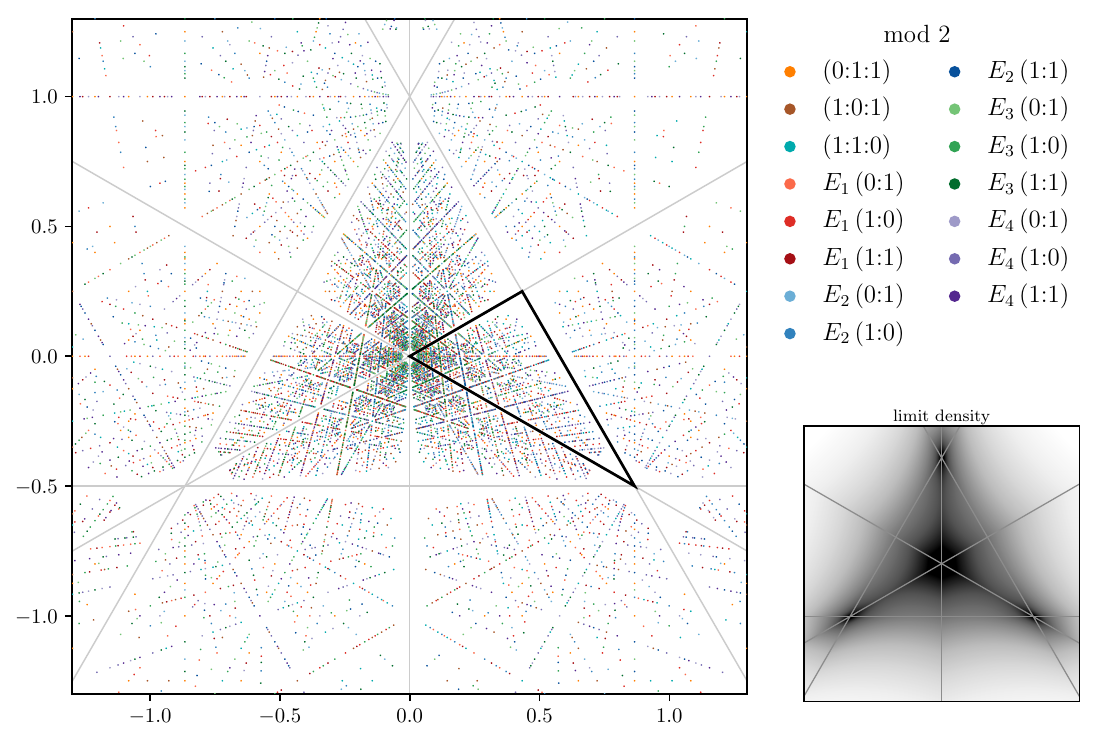}
    \end{center}
    \caption{Rational points of height at most $500$ on a del Pezzo surface $X$ of degree $5$ with a subset of $X(\Rd)$ in the shape of a triangle and the density function for the limit probability measure.
}\label{fig:dp5}
\end{figure}

Several lines in $\Pd^2$ seem to contain particularly many points. These are not accumulating subsets (contributing an order of magnitude of $B$ points, a factor $(\log B)^4$ below the total), but as strict transforms of lines through one of the centers, they have particularly low anticanonical degree $2$ and thus provide particularly good approximants to the points they contain in the spirit of McKinnon's work~\cite{MR2274515} and are \emph{locally accumulating} in a related notion developed by Huang~\cite[Déf~2.1]{MR3667497}. The same holds true for the one-parameter family of conics through all four points.

\subsection{Other approaches to equidistribution}

When dealing with larger classes of varieties, a trick can help with the finite places: congruence conditions can be modeled by linear changes of variables, and if the counting theorem at hand is general enough to treat the varieties resulting from such transformations, only archimedean places remain to be treated; in the context of applications of the circle method, a straightforward approach to this is to modify the smooth weight functions that the method tends to involve anyways. This way, Peyre~\cite[\S 5.5]{Pey95} deduces equidistribution for hypersurfaces in large dimensions based on Birch's seminal work~\cite{Birch}, while Heath-Brown and Loughran~\cite[Theorem~1.5]{HBL25} provide a more recent instance of this method.
More concretely, in the special case of toric varieties, lifting to universal torsors has been used to obtain equidistribution theorems in a rational~\cite{ZH21} and motivic~\cite{Faisant} setting, albeit not using Peyre's above philosophy.
Finally, a converse approach is to establish equidistribution directly (for example using ergodic-theoretic methods as in \cite{GMO} for wonderful compactifications of certain algebraic groups) and to derive Manin's conjecture from this.

\subsection{Structure of the article}

Section~\ref{sec:equidistribution} is devoted to proving Theorem~\ref{thm:equidistribution-theorem} and various corollaries for special situations. In Section~\ref{sec:dp5} we add congruence conditions to the counting problem on split del Pezzo surfaces of degree 5 and use Corollary~\ref{cor:equidistribution-by-counting-mod-q} to obtain equidistribution of rational points (Theorem~\ref{thm:equidistribution_dP5}). The appendix contains the converse of a very general equidistribution theorem by Chambert--Loir and Tschinkel~\cite[Prop.~2.10]{CLT10} that is probably well-known but which we could not find in the literature.

\subsection{Notation and conventions}

Throughout this article, $K$ denotes a number field and $\of_K$ its ring of integers. Write $\Omega_K$ for its set of places and $\infty_K$ for the infinite places. For each $v\in \Omega_K$, equip the local field $K_v$ with an absolute value that is normalized so that $\df x/|x|_v$ is a Haar measure on $K_v^\times$ and the product formula holds; write $K_v^1 = \{x\in K_v : |x|_v=1\}$.

For a finite place $v$, let $\of_v$ denote the ring of integers of $K_v$. Let $\Ab_K$ be the ring of adèles, $\Ab_{\of_K}=\prod_{v\mid\infty}K_v \times \prod_{\pf}\of_\pf$ the ring of integral adèles, and $\Ib_K = \Gm(\Ab_K)= \prod'_{v\in \Omega_K} (K_v^\times, \of_v^\times)$ the group of idèles. Write
\begin{equation*}
    |\cdot| \colon \Ib_K \to \Rd_{>0},\quad (x_v)_{v\in\Omega_K} \mapsto \prod_{v \in \Omega_K} |x_v|_v,
\end{equation*}
and $\Ib_K^1 = \ker(|\cdot|)$ for idèles of absolute value $1$. 

For a finite set $S$ of places, let $\of_S$ be the ring of $S$-integers, $\Ab_K^S$ the adèles off $S$, and $K_S = \prod_{v\in S}K_v$. Write $|\cdot|\colon K_S\to \Rd_{\ge 0}$ for the product of absolute values and $K_S^1 = \{(x_v)_v \in K_S : \prod_{v \in S} |x_v|_v=1\}$ for the group of elements of absolute value $1$.

We use the asymptotic notation $o(1)$ for a quantity with limit $0$ as $B \to \infty$, where the rate of convergence may depend on all other parameters involved (in particular on $K$, $X$, $\Pc'$, and $\qf$).

\subsection*{Acknowledgements}

UD was supported by the Deutsche Forschungsgemeinschaft (DFG) -- 512730679 (RTG2965). 
FW was supported by the Deutsche Forschungsgemeinschaft (DFG) -- 398436923 (RTG2491).

\section{An abstract equidistribution theorem}\label{sec:equidistribution}

In this section, we prove Theorem~\ref{thm:equidistribution-theorem} and its corollaries. The first step is to reduce Theorem~\ref{thm:equidistribution-theorem} to the case where $X=\Pd^n$.

Without loss of generality, assume that $L$ is very ample by replacing it with $dL$, the variety $X$ by $f(X)\subset \Pd^n$ with $n=|\Pc|-1$, and the set $V$ by the multiset $f(V)$; note that the multiplicities are finite as~\eqref{eq:asymptotic-hypothesis} implies the Northcott property for $V$, hence $f(V)$. To prove the equidistribution theorem, we moreover pass to the affine cone of $X\subset \Pd^n$, obtain an equidistribution theorem on it, and then push down all involved measures to obtain Theorem~\ref{thm:equidistribution-theorem}. On projective space, the line bundle $L$ now becomes $\Oc(1)$ and the set $\Pc$ corresponds to the basis $x_0,\dots,x_n$ of $\Hr^0(\Pd^n,\Oc(1))$ inducing the standard height function. Thus, replacing $\tau_{\Pc'}$ by $f_*\tau_{\Pc'}$ (which preserves~\eqref{eq:tamagawa-relative-heights} as $H_{\Pc}(x) = H_{\Pd^n}(f(x))$), we may assume that $X=\Pd^n$ and $\Pc = \{x_0,\dots,x_n\}$.

\subsection{Cones over adelic points}

For an integer $n\ge 1$, we begin by studying projective space $\Pd^n_K$ and its affine cone $C = \Ad_K^{n+1}\setminus \{0\}$. The action of $\Gm$ on the affine cone induces an action of $\Ib_K$ and its subgroup $\Ib_K^1$ on $C(\Ab_K)$.
Write $A$ for the quotient $C(\Ab_K)/\Ib_K^1$.
As $K^\times \subset \Ib_K^1$, the map
\begin{equation*}
    i\colon \Pd^n(K) \to A,\quad (x_0:\dots:x_n)\mapsto ((x_0,\dots, x_n))_{v\in\Omega_K}
\end{equation*}
is well-defined.
Let $\pi\colon A \to \Pd^n(\Ab_K)$ be the projection. Moreover, for each place $v$, let
\begin{equation*}
    \norm{\cdot}_v\colon K_v^{n+1}\to \Rd_{>0},\quad (x_0,\dots,x_n)\mapsto \max\{|x_0|_v,\dots,|x_n|_v\}
\end{equation*}
be the supremum norm; their product factors through $A$:
\begin{equation*}
    \norm{\cdot}\colon A \to \Rd_{>0},\quad (x_v)_v\mapsto \prod_v \norm{x_v}_v.
\end{equation*}
These two maps induce an isomorphism $A\cong \Pd^n(\Ab_K) \times \Rd_{>0}$; its inverse maps an adelic point on $\Pd^n$ and a constant $c>0$ to the unique representative $\tilde \xb\in A$ with $\norm{\tilde \xb}=c$. We shall moreover be interested in the subspace $A^\br\cong \Pd^n(\Ab_K)\times (0,1)$ of points of norm smaller than $1$. 

Write $A_S = \prod_{v\in S} C(K_v) / K_S^1$ (with $A_S\cong \Pd^n(K_S) \times \Rd_{>0}$) to get a decomposition
\begin{equation}\label{eq:decomp-A-AS}
    \begin{tikzcd}[row sep = tiny]
        A_S \times \Pd^n(\Ab_K^S) \arrow[r,"\sim"] & A, \\
        ((\yb^v)_{v\in S} K_S^1,\ (\xb^v)_{v\not\in S}) \arrow[r,mapsto] & (\yb^v)_{v\in\Omega_K} \Ib_K^1,
    \end{tikzcd}
\end{equation}
where $\yb^v$ is a lift of $\xb^v$ to $C(K_v)$ with $\norm{\yb^v}_v=1$ for each $v\not\in S$, which restricts to an isomorphism $A_S^\br\times \Pd^n(\Ab_K^S)\cong A^\br$, where $A_S^\br$ is defined analogously to $A^\br$. The projection to the first components maps $\ybb \Ib_K^1\in A$ to $(\tilde{\yb}^v)_{v\in S}$, where $\tilde{\ybb}$ is a representative of $\ybb \Ib_K^1$ such that $\norm{\tilde{\yb}^v}_v = 1$ for all $v\not\in S$. 
For each $v\in S$, define
\begin{equation*}
   A_v^\br = \begin{cases}
    \{ \yb\in C(K_v) : |\yb|_v<1\}/K_v^1 & \text{if $v$ is archimedean and} \\
    \{ \yb\in C(K_v) : |\yb|_v\le 1\}/K_v^1 & \text{otherwise.}
    \end{cases}
\end{equation*}
Note that $\prod_{v\in S}A_v^\br \to A_S^\br$ is surjective. If $\Pc'\subset \Hr^0(\Pd^n, \Oc(1))$ is a generating set of sections, let
\begin{equation*}
    H_{\Pc',v}\colon K_v^{n+1}\setminus \{0\} \to \Rd_{>0},\quad \xb\mapsto \max_{s\in \Pc'}|s(\xb)|_v.
\end{equation*}
The product $H_{\Pc'} = \prod_v H_{\Pc', v}\colon C(\Ab_K)\to \Rd_{>0}$ is well-defined as almost all factors are $1$ for given $\xb\in C(\Ab_K)$ and factors through $A$.
The composition with $i$ is the height function $H_{\Pc'}$ associated with the norm on $\Oc(1)$ arising from pulling back the standard sup-norm along the morphism $\Pd^n\to \Pd^{|\Pc'| - 1}$ induced by $\Pc'$.

\subsection{Measures}

Let $\mu$ be a measure on $\Pd^n(\Ab_K)$ and $H$ be a height function. Define a measure $\mu^H$ on $A \cong \Pd^n(\Ab_K)\times \Rd_{>0}$ by setting
 \begin{equation}\label{eq:cone-measure-adeles}
    \df \mu^H = \ind_{\{H \le 1\}} \frac{a H(\xb)^a}{\norm{\xb}} \df \mu(\pi(\xb)) \df\norm{\xb}.
 \end{equation}
 
\begin{lemma}\label{lem:pushforward-measure-cone}
    Then $\pi_*\mu^H = \mu$.
\end{lemma}

\begin{proof}
    Let $W\subset \Pd^n(\Ab_K)$ be measurable and $\pi^{-1}(W)$ be its cone in $A$. Then
    \begin{equation*}
        \int_{\yb \in \pi^{-1}(W)} \df \mu^H
        = \int_{\xb\in W}\int_{\substack{t\in \Rd_{>0}\\ H(t \tilde \xb)\le 1}}
        \frac{a H(t\tilde \xb)^a}{t}\df \mu \df t,
    \end{equation*}
    where $\tilde \xb$ is, as before, a lift of $\xb\in W\subset C(\Ab_K)$ to $A$ with $\norm{\tilde \xb}=1$. Indeed, note that $H(t\tilde \xb)\le 1$ is equivalent to $t\le 1/H(\tilde \xb)$, and $H(t\tilde \xb)^a = t^a H(\tilde\xb)^a$ as $H$ is homogeneous of degree $1$, being induced by linear polynomials.
\end{proof}

 Let $\tau = \tau_H$ be the measure associated with the standard height function $H(\xb) = \prod_v\norm{\xb}$ and $\tau_A = a \norm{\yb}^{a-1}\df \mu \df \norm{\yb}$ its lift to $A$ without the indicator function in~\eqref{eq:cone-measure-adeles}. If $H_1$ and $H_2$ are further height functions, then
\begin{equation}\label{eq:difference-of-measures}
    \df \tau_{H_1} - \df \tau_{H_2} = \pi_*\left((\ind_{\{H_1\le 1\}} - \ind_{\{H_2\le 1\}})\df \tau_A  \right)
\end{equation}
by Lemma~\ref{lem:pushforward-measure-cone} and~\eqref{eq:cone-measure-adeles}.

\subsection{A basis of the affine cone}

The remaining crucial ingredient for equidistribution is to construct indicator functions of a basis of its topology as linear combinations of the indicator functions appearing in~\eqref{eq:difference-of-measures}. Let $v$ be a place of $K$.
The sets
\begin{equation}\label{eq:basis-epsilon}
    U_{\yb,\epsilon} = \left\{(x_0,\dots, x_n)\in K_v^{n+1} : 
    \substack{1-\epsilon< \norm{\xb}_v/\norm{\yb}_v<1+\epsilon,\\
    |x_iy_j-x_jy_i|_v<\epsilon 
    \text{ for all }0\le i,j\le n }\right\},
\end{equation}
where $\yb$ ranges over $A_v^\br$ and $\epsilon$ over small positive real numbers, form a basis of the topology of $A_v^\br$ by approximating the norm and direction of an arbitrary point $\yb$.  (For finite places and sufficiently small $\epsilon$, note that the two norms $\norm{\xb}=\norm{\yb}$ are equal and achieved by the same index $k$ through the ultrametric inequality; then normalize by $c=y_k/x_k\in K_v^1$ to get $|cx_i-y_i|<\epsilon/\norm{\yb}$ for all $i$.) Their products thus form a basis of the topology of $\prod_{v\in S}A_v^\br$, and the images of such products one of $A_S^\br$.

\subsection{Some geometry of numbers}

The next step is to approximate these sets with a variant that draws its coefficients from $\of_S$.

\begin{lemma}\label{lem:geometry-of-numbers}
    Let $S$ be a finite set of places containing the archimedean ones. Let $(y^v)_v \in \prod_{v\in S} K_v^n\setminus \{0\}$, and let $\epsilon$ and $\kappa$ be positive real numbers. Then there is a positive real number $\kappa_v$ for each $v\in S$ such that the following holds true.
    \begin{enumerate}
        \item There are $S$-integers $d_1$ and $d_2$ such that
        \begin{itemize}
            \item $|d_1|_v,|d_2|_v>1$ for all $v\in S$,
            \item $(1-\epsilon) \norm{y^v} < \kappa_v/|d_1|_v < \norm{y^v} < \kappa_v/|d_2|_v < (1+\epsilon)\norm{y^v}$ for every archimedean $v\in S$, and
            \item $\kappa_v/|d_1|_v = \kappa_v/|d_2|_v = \norm{y^v}$ for every finite $v\in S$.
        \end{itemize}
        \item For every $0\le i,j\le n$, there are $u_{i,j}$ and $w_{i,j}$ in $\of_S$ such that
        \begin{itemize}
            \item $\max\{|u_{i,j}|_v,|w_{i,j}|_v\}>\kappa\kappa_v$ for all $v\in S$ and
            \item $|u_{i,j}y^v_i-w_{i,j}y^v_j|_v<\kappa_v$ for all $v\in S$.
        \end{itemize}
    \end{enumerate}
\end{lemma}

\begin{proof}
The first claim is clear as soon as $\kappa_v$ is sufficiently large (and an absolute value if $v$ is finite); the Lemma will thus follow once we have proved the second claim for all sufficiently large $\kappa_v$.
Fix $i$ and $j$ for now, and let
\begin{equation*}
    U^v = U^v_{Q,R}= \left\{(u_{i,j},w_{i,j})\in K_v^2 :
    \substack{
        |u_{i,j}y^v_i-w_{i,j}y^v_j|_v\le Q,\\
        \kappa_v Q \le \norm{(u_{i,j},w_{i,j})}\le \kappa_v Q+R 
    }\right\}
\end{equation*}
and $U=U_{Q,R} = \prod_{v\in S}U^v_{Q,R}$.
Our aim is to guarantee the existence of a lattice point in $U$ for all sufficiently large $Q$ and $R$. This will follow once we establish that the Haar measure of $U_{Q,R}$ grows more quickly than that of its boundary, see e.g.~\cite[Lem.~3.2]{MR4248638}.
Concretely, we study the $F$-boundary of $U_{Q,R}$ for compact sets $F$, that is
\begin{equation*}
    \partial_{F} U_{Q,R} = (\overline{U_{Q,R}}+F)\setminus U_{Q,R}^\circ \cup (\overline{U_{Q,R}^c} +F) \setminus (U_{Q,R}^c)^\circ.
\end{equation*}
After applying a linear change of variables, shrinking the sets, and appropriately modifying $Q$ and $R$, we may assume
\begin{equation*}
    U^v = \{(u,w)\in K_v^2 : |u|_v\le Q,\ \kappa Q \le |w|_v \le \kappa Q+R \}.
\end{equation*}
In particular, $U^v$ has volume $\gg_v RQ$ for sufficiently large $Q$ and $R$, so that $U$ has volume $\gg (QR)^s$ where $s = |S|$.

In order to leverage the identity
\begin{equation}\label{eq:boundary-product}
    \partial_{F_1\times F_2}(U_1\times U_2) = \big(\partial_{F_1}U_1 \times (U_2+F_2)\big) \cup \big((U_1+F_1)\times \partial_{F_2}U_2\big)
\end{equation}
for the boundary of a product, we shall work with boxes $F_v = B_1(0,r)^2 \subset K_v^2$ of ``radius'' $r\in \im|\cdot|_v$ around the origin. 

If $v$ is archimedean, then $F_v+U^v$ has volume $O((Q+r)(R+r)) = O_r(QR)$, while  $\partial_{F_v} U^v$ has volume $O(rQ+rR+r^2) = O_r(Q+R)$ based on~\eqref{eq:boundary-product}.
If $v$ is finite, then $\partial_{F_v} U^v=\emptyset$ as soon as $\min\{Q,\kappa Q\}>r$, as $U^v+F_v = U^v$ and $(U^v)^c + F_v = (U^v)^c$ by the ultrametric inequality.

Setting $F = \prod_{v\in S} F_v$ and appealing to~\eqref{eq:boundary-product} again, the $F$-boundary of the product set $U$ has a volume of at most
$\mu(\partial_F U) = O_r((QR)^{s-1}(Q+R))$. It follows that 
\begin{equation*}
    \frac{\mu(\partial_F U_{Q,R})}{\mu(U_{Q,R})} \ll_r \frac{(QR)^{s-1}(Q+R)}{(QR)^s}\to 0
\end{equation*}
for fixed $F$, as $Q,R \to \infty$. As every compact set is eventually contained in one of the sets $F$, the assertion follows.
\end{proof}

\subsection{A basis of the adelic affine cone}

Let $\Lc = \{l_1,\dots, l_s\}$ be a set of linear forms in $\of_S\langle x_0,\dots, x_n\rangle$, let $\db = (d_1,d_2)\in \of_S^2$ be a tuple of $S$-integers with $|d_1|_v\ge|d_2|_v \ge 1$ for all $v\in S$ such that the inequality $|d_1|_v>|d_2|_v$ is strict for all archimedean $v$, and let $\kappa >0$ be a real number decomposed into $\kappa = \prod_{v\in \Omega_K} \kappa_v$ with $\kappa_v>0$ and $\kappa_v = 1$ for all $v\not\in S$. Given these data and $i\in \{1,2\}$, define a set
\begin{equation*}
    \Pc_{\db,\Lc,i} = \{x_0,\dots, x_n, d_ix_0,\dots,d_ix_n, l_1,\dots, l_s\},\\
\end{equation*}
of polynomials and local factors $H_{v,\db,\Lc,\kappa_v,i}\colon C(K_v)\to \Rd_{>0}$ for 
\begin{equation*}
 H_{v,\db,\Lc,\kappa_v,i} (\xb) =  
        \kappa_v^{-1}\max\left\{ \substack{
        |x^v_0|,\dots,|x^v_n|,\\ |d_ix^v_0|,\dots, |d_ix^v_n|, \\
        |l_1(\xb^v)|,\dots, |l_s(\xb^v)|
    }\right\}
\end{equation*}
of the height functions $H_{\db,\Lc,\kappa,i}\colon A \to \Rd_{>0}$ with
\begin{equation*}
    H_{\db,\Lc,\kappa,i} = \kappa^{-1}H_{\Pc_{\db,\Lc,i}} = \prod_v H_{v,\db,\Lc,\kappa_v,i}.
\end{equation*}
These two height functions factor through $\pi_S\colon A\to A_S$ via
\begin{equation*}
    H_{S,\db,\Lc,\kappa,i} = \prod_{v\in S}H_{v,\db,\Lc,\kappa_v,i};
\end{equation*}
indeed, note that $H_{v,\db,\Lc,\kappa_v,i}(\yb^v) = \norm{\yb^v}$ for all $\yb^v\in C(K_v)$ by the ultrametric inequality and the restrictions on the coefficients of the linear forms and the $d_i$, and that this norm is one for the lifts $\yb^v$ in~\eqref{eq:decomp-A-AS}.
Based on this, let
\begin{equation*}
    U_{S,\db,\Lc,\kappa} = \{H_{S,\db,\Lc,\kappa,1} < 1\} \setminus \{H_{S,\db,\Lc,\kappa,2} \le 1\} \subset A_S,
\end{equation*}
and write 
\begin{equation*}
     \overline{U}_{S,\db,\Lc,\kappa} = \{H_{S,\db,\Lc,\kappa,1} \le 1\} \setminus \{H_{S,\db,\Lc,\kappa,2} <1 \}
\end{equation*}
for the closures of the latter sets. Then $U_{\db,\Lc,\kappa} = \pi_S^{-1}(U_{S,\db,\Lc,\kappa})$ satisfies
\begin{equation*}
    U_{\db,\Lc,\kappa} = \{H_{\db,\Lc,\kappa,1} < 1\} \setminus \{H_{\db,\Lc,\kappa,2} \le 1\} \subset A
\end{equation*}
and $\overline{U}_{\db,\Lc,\kappa} = \pi_S^{-1}(\overline{U}_{S,\db,\Lc,\kappa})$ can be described analogously.
We shall test for convergence on the family
\begin{equation*}
    \Uc_S = 
    \left\{
        U_{S,\db,\Lc,\kappa}:
        \substack{
            \db=(d_1,d_2)\in \of_S^2 \text{ with } |d_1|_v\ge |d_2|_v\ge 1 \text{ for all } v\in S,\\ \Lc\subset \of_S\langle x_0,\dots,x_n\rangle,\ \kappa\in \Rd_{>0}
        }
    \right\}
\end{equation*}
of open subsets of $A^\br_S$ and its closure $\widetilde{\Uc}_S$ under finite intersections by means of its preimage
\begin{equation*}
    \Uc = 
    \left\{
        U_{\db,\Lc,\kappa}:
        \substack{
            \db=(d_1,d_2)\in \of_S^2 \text{ with } |d_1|_v\ge |d_2|_v\ge 1 \text{ for all } v\in S,\\ \Lc\subset \of_S\langle x_0,\dots,x_n\rangle,\ \kappa\in \Rd_{>0}
        }
    \right\}
\end{equation*}
in $A^\br$ and its closure $\widetilde{\Uc}$ under finite intersections.

\begin{lemma}\label{lem:lemma-product-of-height-annuli}
    These sets satisfy
    \begin{equation*}
        \overline{U}_{\db,\Lc,\kappa} = \left(\prod_{v\in \Omega_K} (\{H_{v,\db,\Lc,\kappa_v,1}\le 1\} \setminus \{H_{v,\db,\Lc,\kappa_v,2}<1\})\right) \Ib_K^1 \subset A
    \end{equation*}
    and analogously for $\overline{U}_{S,\db,\Lc,\kappa}$.
\end{lemma}

\begin{proof}
    The inclusion of the right-hand side in the left-hand side is clear. For the other inclusion, take a point $\xbb$ in the right-hand side. As $H_1(\xbb)^{-1}\ge 1\ge H_2(\xbb)^{-1}$, there is an idèle $\ub$ of norm 1 such that $|u^v|_v = H_{1,v}(\xb^v)^{-1} = H_{2,v}(\xb^v)^{-1}$ for ultrametric places and $H_{1,v}(\xb^v)^{-1} \ge |u^v|_v \ge  H_{2,v}(\xb^v)^{-1}$ for archimedean ones. Now $\ub\xbb$ lies in the left-hand side.
\end{proof}

\begin{lemma}
    The family $\Uc_S$ is a topological basis of $A_S^\br$, and hence so is $\widetilde{\Uc}_S$.
\end{lemma}

\begin{proof}
For each $v\in S$, let $\yb^v = (y^v_0,\dots,y^v_n)\in A_v^\br$, and let $\epsilon>0$. 
Appealing to Lemma~\ref{lem:geometry-of-numbers}, 
we may find $S$-integers $u_{i,j}$ and $w_{i,j}$ for each $1\le i,j\le n$ and a positive real $\kappa_v$ for each $v\in S$ such that
\begin{align*}
    |u_{i,j}y_i - w_{i,j}y_j|_v &< \kappa_v \quad \text{and}\\
    \kappa_v\norm{(y_0,y_i)}_v/\epsilon &< \norm{(u_{i,j},w_{i,j})}_v
\end{align*}
for all $1\le i,j\le n$ and $v\in S$, and
we may find $S$-integers $d_1$ and $d_2$ with
\begin{enumerate}
    \item $(1-\epsilon) \norm{\yb^v}_v < \kappa_v/|d_1|_v<\norm{\yb^v}_v < \kappa_v/|d_2|_v< (1+\epsilon)\norm{\yb^v}_v$ for all archime\-dean $v\in S$,
    \item $\kappa_v/|d_1|_v = \norm{\yb^v}_v$ and $\norm{d_1}_v=\norm{d_2}_v$ for all ultrametric $v\in S$.
\end{enumerate}
Set $\kappa_v=1$ for $v\not\in S$ and $\kappa=\prod_{v\in S}\kappa_v$, and $\Lc = \{u_{i,j}x_i - w_{i,j}x_j\}$.
For each $v$, the set $\{H_{v,\db,\Lc,\kappa,1}\le 1\} \setminus \{H_{v,\db,\Lc,\kappa,2}< 1\}$ contains $\yb$ in its interior and is contained in $U_{\yb^v,2\epsilon}$ as in~\eqref{eq:basis-epsilon}. Appealing to Lemma~\ref{lem:lemma-product-of-height-annuli}, the interior $U_{S,\db,\Lc,\kappa}$ of their product in $A_S$ is thus contained in $\prod U_{\yb^v,2\epsilon}K_S^1$, and $\Uc_S$ forms a basis.
\end{proof}

\subsection{Proof of the equidistribution theorem}

We are now ready to prove Theorem~\ref{thm:equidistribution-theorem}; recall that it suffices to do so for $X=\Pd^n$. To do so, it is convenient to define $N'(\Pc',B) = |\{x\in V : H_{\Pc'}(x)< B\}|$; note that it also satisfies the asymptotic expansion~\eqref{eq:asymptotic-hypothesis}.
Let $i\colon \Pd^n(K)\to A$ map a rational point $x$ to the unique preimage of itself regarded as an adelic point with $\norm{x} = H(x)$, and let $i_S = \pi_S\circ i\colon \Pd^n(K)\to A_S$.
Let $\mu_B$ be the probability measure $N'(\Pc,B)^{-1}\sum_{x\in V, H(x) < B} \delta_{i_S(x)/B}$. Note that it is supported on $A_S^\br$. Moreover, let $\tau$ be the measure on $\Pd^n(\Ab)$ associated with the height function $H = H_{\Pc} = \norm{\cdot}$.
\begin{lemma}
    Let $U_S\in \Uc_S$ and $U = \pi^{-1}(U_S)\in \Uc$. Then $\mu_B(U_S)\to \tau^H(U)/\tau^H(A^\br)$ as $B\to \infty$.
\end{lemma}

\begin{proof}
    Let $U_S = U_{S,\db,\Lc,\kappa}$. Then $i_S(x)/B\in U$ if and only if $i(x)/B\in U_{\db,\Lc, \kappa}$, that is, if and only if
    $H_{\db,\Lc,\kappa,1}(i(x)/B)<1<H_{\db,\Lc,\kappa,2}(i(x)/B)$. As  $H_{\db,\Lc,\kappa,1}(i(x))<H_{\db,\Lc,\kappa,2}(i(x))$ holds automatically and $H_{\db,\Lc,\kappa,1}(i(x)/B)<1$ implies $H(i(x)/B)<1$, that is, $H(x)<B$, we get
    \begin{align*}
        \mu_B(U) &= \frac{N'(\Pc_{\db,\Lc,1},\kappa B) - N(\Pc_{\db,\Lc,2},\kappa B)}{N'(\Pc,B)} \\
        & =
        \left( \int_{X(\Ab_K) }\df\tau_H\right)^{-1}
        \int_{X(\Ab_K)} \kappa^{-a} (\df\tau_{H_1}-\df\tau_{H_2})(1+o(1)),
    \end{align*}
    where $H_i$ is the height function associated with $\Pc_{\db,\Lc, i}$ for $i\in \{1,2
    \}$.
    Appealing to~\eqref{eq:cone-measure-adeles} and~\eqref{eq:difference-of-measures}, the right-hand side becomes
    \begin{equation*}
       \tau^H(A^\br)^{-1}
        \int_{A} \kappa^{-a}\pi_*(\ind_{H_1\le 1} - \ind_{H_2\le 1})\df\tau_A,
    \end{equation*}
    and this last integral coincides with $\tau^H(U)$ upon the linear change of variables multiplying the second factor of $A\cong \Pd^n(\Ab_K)\times \Rd_{>0}$ by $\kappa$ changing the measure $\tau^H$ by a factor $\kappa^a$ as per its definition~\eqref{eq:cone-measure-adeles} (note that $H$ is homogeneous of degree $1$ as an $\Oc(1)$-height).
\end{proof}

\begin{lemma}
    \begin{enumerate}
        \item Every $U\in \Uc$ is a $\tau^H$-continuity set.
        \item Let $U\in \widetilde{\Uc}$ and $\epsilon>0$. Then there are $U_2,U_1\in \Uc$ with $U_2\subset U \subset U_1$, $\tau^H(U\setminus U_2)<\epsilon$, and $\tau^H(U_1\setminus U)<\epsilon$.
        \item  Let $U_S\in \widetilde{\Uc}_S$. Then $\mu_B(U_S)\to \tau^H(U_S)$ as $B\to \infty$.
    \end{enumerate}
\end{lemma}
\begin{proof}
   For the first statement, let $U = U_{\db, \Lc, \kappa}$. Its boundary is contained in (the image in $A_S$ of) the union over hypersurfaces defined by $|l(\yb^v)|_v=\kappa_v$, where $l$ is one of the forms in $\Lc$ or $d_iy_j$ for $i\in\{1,2\}$ and $0\le j\le n$. In particular, if the boundary were not a null set, one of these hypersurfaces would need to have positive measure. Writing $\ybb(\xbb, h)$ for the unique representative $\ybb\in A$ of $\xbb$ with $\norm{\ybb}=h$, this volume is
   \begin{equation*}
     \int_{h\in (0,1)} a h^{a-1}\int_{\xbb\in X(\Ab_K),\ l(\ybb(\xbb,h))=\kappa} \df \tau \df h,
   \end{equation*}
   so in particular, the inner integral needs to be positive for uncountably many $h$. But the sets $\{l(\ybb(\xbb, h)) = \kappa\}$ are disjoint for different $h$ as $l(\ybb(\xbb, h)) = hl(\ybb(\xbb,1))$ and $\kappa\ne 0$. Hence
   \begin{equation*}
    \tau(\Pd^n(\Ab_K)) \ge \sum_{h\in (0,1)}\int_{\xbb\in \Pd^n(\Ab_K),\ l(\ybb(\xbb,h))=\kappa} \df \tau =\infty,
   \end{equation*}
   a contradiction to the assumptions on $\tau$.

   It suffices to prove the second statement for binary intersections $U\cap U'$ of sets in $\Uc$. The general case follows from repeated applications. To this end, let $U = U_{\db,\Lc,\kappa}$ and $U' = U_{\db',\Lc',\kappa'}$, and let $\epsilon > 0$.
   After multiplying the sets of parameters with suitable $e$ and $e'$, we may assume that $1-\epsilon < \kappa/\kappa'<1+\epsilon$; after changing the decompositions of $\kappa$ and $\kappa'$ as products of $\kappa_v$ and $\kappa_v'$, we may assume the same for these factors, all without changing $U$ or $U'$. Further increasing $e$ and $e'$, we may assume that the absolute values $|d_i|_v$ and $|d_i'|_v$ are arbitrarily large. We may thus find $d_i^{(j)}\in \of_S$ for $i,j\in \{1,2\}$ such that 
   \begin{align*}
    (1-\epsilon)\max\{|d_1|_v,|d_1'|_v\} &< |d^{(2)}_1|_v < \max\{|d_1|_v,|d_1'|_v\}, \\
    \max\{|d_1|_v,|d_1'|_v\} &< |d^{(1)}_1|_v < (1+\epsilon)\max\{|d_1|_v,|d_1'|_v\},
   \end{align*}
   and analogously for $d_2^{(i)}$ in reversed order (so that $d_2^{(1)}< d_2^{(2)}$) and with the maxima replaced by minima. Now define
   \begin{equation*}
    U_{1,\epsilon} = U_{(d_1^{(1)},d_2^{(1)}),\Lc\cup \Lc',\max\{\kappa,\kappa'\}} \quad \text{and}\quad
   U_{2,\epsilon} = U_{(d_1^{(2)},d_2^{(2)}),\Lc\cup \Lc',\min\{\kappa,\kappa'\}}.
   \end{equation*}
   As $\epsilon\to 0$ these are contained in each other and satisfy $U_{2,\epsilon}\subset U\cap U'\subset U_{1,\epsilon}$,
    $\bigcup_{\epsilon>0} U_{2,\epsilon} = U\cap U'$, and $\bigcap_{\epsilon>0} U_{2,\epsilon} = \overline{U\cap U'}$. As the boundary of the last set is contained in the union of those of $U$ and $U'$, hence is a null set by the first statement, the second statement follows.
    The final statement now follows from the second.
\end{proof}

\begin{prop}\label{prop:equidistribution}
    The probability measures $\mu_B$ converge weakly to $(\pi_S)_*\tau^H / \tau^H(A)$.
\end{prop}

\begin{proof}
The family $\widetilde{\Uc}$ is a $\pi$-system and a basis, and $A^\br$ is separable (rational vectors form a countable dense set, for instance). Weak convergence now follows from~\cite[Thm.~2.3]{MR1700749}.
\end{proof}

\begin{proof}[Proof of Theorem~\ref{thm:equidistribution-theorem}]
    Let $\pi\colon A_S \to \prod_{v\in S} \Pd^n(K_v)$. Then the abstract equidistribution theorem amounts to saying that $\pi_*\mu_B$ converges weakly to $\tau = \pi_* \tau^H$, which follows from Proposition~\ref{prop:equidistribution}.
\end{proof}

\subsection{Corollaries}

With the main theorem at hand, we turn to several corollaries, for the proof of which we drop the assumption that $X=\Pd^n$.

\begin{proof}[Proof of Corollary~\ref{cor:adelic-equidistribution}]
    The arguments to this point prove this corollary by simply skipping the projection $A\to A_S$. Alternatively, Theorem~\ref{thm:equidistribution-theorem} proves convergence on continuity sets of the form $\prod_v W_v$ such that almost all $W_v$ coincide with $X(K_v)$. As this family is closed under finite intersections and generates the Borel $\sigma$-algebra, weak convergence follows.
\end{proof}

\begin{cor}\label{cor:equidistribution-by-counting-mod-q}
    Keep all notation from Theorem~\ref{thm:equidistribution-theorem}; let $\Xf$ be a proper model of $X$, and let $\Lf$ be an $\of_S$-model of $dL$ that is globally generated by $\Pc\subset \Hr^0(\Xf,\Lf)\subset\Hr^0(X,dL)$. For a point $r\in \Xf(\of_S/\qf)$ modulo a nonzero ideal $\qf\subset \of_S$, let 
    \begin{equation*}
        N(\Pc',r, B) = |\{x\in V: H_{\Pc'}(x)\le B,\ \congr{x}{r}{\qf}\}|
    \end{equation*}
    and
    \begin{equation*}
        c_r = \lim_{l\to \infty} 
        \frac{
            |\{r'\in \Xf(\of_S/\qf^l):\congr{r'}{r}{\qf}\}|
        }{
            |\Xf(\of_S/\qf^l)|
        }.
    \end{equation*}
    Assume that
    \begin{equation}\label{eq:tamagawa-like-assumption}
        \tau_{\Pc'}(\{x\in\Xf(\Ab_{\of_K}) : \congr{x}{r}{\qf} \}) = c_r\tau_{\Pc'}(\Xf(\Ab_K))
    \end{equation}
    and
    \begin{equation*}
        N(\Pc', r, B) =
        c c_r \tau_{\Pc'}(X(\Ab_K))
        B^a (\log B)^{b-1} (1+o(1))
    \end{equation*}
    for all ideals $\qf$ and points $r$ as above and all $\Pc'$ as in Theorem~\ref{thm:equidistribution-theorem}.
    
    Then $f(V)$ is equidistributed in $f(X)(\Ab_K)$ with limit measure $f_*\tau_\Pc$ when ordered by $H_{\Pc}$.
\end{cor}

\begin{remark}\label{rmk:salberger-model-measure}
    Note that $c_r = |\Xf(\of_S/\qf)|^{-1}$ if $\qf$ is coprime to all primes of bad reduction by Hensel's Lemma and the Chinese Remainder Theorem. Moreover, note that condition~\eqref{eq:tamagawa-like-assumption} holds for the Tamagawa measure $\tau_{\Pc'}$~\cite[Thm.~2.14]{Sal98} on a smooth variety $X$ if $\Pc'$ globally generates the anticanonical sheaf of the model, as $\tau_{\Pc'}$ is the model measure in that case.
    In the presence of a failure of the relevant adelic approximation property, the measure $\tau_{\Pc}$ should be interpreted as the restriction to the relevant Brauer--Manin set, say; in this case, $c_r$ and \eqref{eq:tamagawa-like-assumption} have to be modified by the relevant Brauer--Manin pairing and only be studied for sufficiently divisible moduli. We believe that a corollary along the lines of the above one formulated in the greatest possible generality would be too verbose to be of much use, however.
\end{remark}

\begin{proof}[Proof of Corollary~\ref{cor:equidistribution-by-counting-mod-q}]
    Write $X(\Ab_K) = X(K_S) \times X(\Ab_K^S)$ and fix one of the standard basis sets $W = \{x\in X(\Ab_K^S) : \congr{x}{r}{\qf}\}$ in the second factor. Now apply Theorem~\ref{thm:equidistribution-theorem} to $V\cap W$ to get equidistribution in the first factor.
\end{proof}

\begin{proof}[Proof of Corollary~\ref{cor:rational}]
This follows directly from Theorem~\ref{thm:equidistribution-theorem} and \cite{BD24}. Note that the theorem involves a gcd condition on the sets of polynomials that essentially prevents modifications to the height at finite places. It is, however, preserved by $\of_K$-linear combinations, hence Theorem~\ref{thm:equidistribution-theorem} is applicable with $S = \infty_K$.
\end{proof}

\begin{proof}[Proof of Corollary~\ref{cor:integral}]
For the log anticanonical model, this follows immediately from Theorem~\ref{thm:equidistribution-theorem} and~\cite{BD25integral}. The log anticanonical morphism contracts the exceptional divisors $E_1$ and $E_2$ meeting $\Df$. They contribute $0$ to the counting function, as points on them are excluded by its definition, and are null sets for $\tau_{D}$, as the latter is concentrated on $D$ and absolutely continuous on it. Equidistribution on $X$ thus follows.
\end{proof}

\section{Equidistribution on split del Pezzo surfaces of degree 5}\label{sec:dp5}

Let $X$ be a smooth split quintic del Pezzo surface over a number field $K$, and let $U$ be the complement of its ten lines. Let $H$ be a height function induced by a metric on the anticanonical bundle on $X$. Let $\tau_H$ be the corresponding Tamagawa measure on $X(\Ab_K)$ as in \cite[D\'efinition~4.6]{Pey03}, and let $\alpha(X)$ be its effective cone constant \cite[\S 4]{Pey03}.

Recall that $p\colon X \to \Pd^2_K$ is the blow-up of the four points $p_1,\dots,p_4$ as in \eqref{eq:4_points}; write $E$ for the exceptional divisor on $X$ and $\Ic$ for the ideal sheaf of $\{p_1,\dots, p_4\}$, so that $p^*\Ic = \Oc_X(-E)$. Then anticanonical global sections on $X$ correspond to cubic forms in $K[Y_1,Y_2,Y_3]$ vanishing in these points via the isomorphism
\begin{equation*}
    p^*(\Ic \Oc_{\Pd^2}(3))
    \cong \Oc_X(3H-E) \cong \omega_X^\vee
\end{equation*}
mapping such a cubic form $P$ to its product with the canonical section $s_E$ of $E$.
Similarly, anticanonical sections on $\Xf$ correspond to cubic forms in $\OK[Y_1,Y_2,Y_3]$ vanishing in them.
We call an anticanonical height function \emph{admissible} if it is induced by global generators of $\omega_{\Xf/\OK}^\vee$ that also generate the six-dimensional space of global sections. If $\Pc$ is the corresponding subset of $\OK[Y_1,Y_2,Y_3]$, its elements satisfy the coprimality condition
\begin{equation}\label{eq:coprimality-polynomial}
    \gcd_{P \in \Pc} P(y) = \frac{\gcd(y_2,y_3)\gcd(y_1,y_3)\gcd(y_1,y_2)\gcd(y_1-y_2,y_1-y_3)}{\gcd(y_1,y_2,y_3)}
\end{equation}
for all $y = (y_1:y_2:y_3) \in \Pd^2(K) \setminus\{p_1,\dots,p_4\}$. Indeed, take a primitive representation of a local point $y=(y_1:y_2:y_3)\in U(K_\pf)$. The latter can reduce to at most one of the centers of the blow-up, say $p_1=(1:0:0)$, in which case the right-hand side of~\eqref{eq:coprimality-polynomial} simplifies to $\gcd(y_2,y_3)$. The image of a polynomial $P\in \Pc\subset \Oc_{\Pd^2}(3)$  in
\begin{equation*}
    \omega_X^\vee \cong (p^*\Oc_{\Pd^2}(3))(-E)
\end{equation*}
 is its product with the canonical section of $E$, hence its corresponding vanishing order in the special fiber $y_\pf$ of $y$ is
\begin{equation}\label{eq:order-of-vanishing-P}
    v_{y_{\pf}}(Ps_E) = v_{y_{\pf}}(P) + v_{y_{\pf}}(s_{E}) = v_\pf(P(y)) - \min\{v_\pf(y_2),v_\pf(y_3)\};
\end{equation}
as this order is $0$ for at least one $P\in \Pc$ by virtue of the family being generating, the $\pf$-primary part of equation~\eqref{eq:coprimality-polynomial} follows.
In particular, every admissible height function is admissible in the sense of~\cite[\S 1.1]{BD24}, and concretely, $H$ can be defined as
\begin{equation*}
    H(x)\coloneq\prod_v \max_{P \in \Pc}|P(y_1,y_2,y_3)|_v
\end{equation*}
for any $x \in U(K)$ with $p(x)=(y_1:y_2:y_3) \in \Pd^2(K)$. 

On the other hand, every closed point can be lifted to a local point of the form $(y_1:y_2:y_3)\in U(L_\Pf)$ over a finite extension $L_\Pf/K_\pf$.
One may readily verify that the right-hand side of~\eqref{eq:order-of-vanishing-P} is $0$ for at least one polynomial in the set $\{P_\sigma, Q_\sigma : \sigma \in S_3\}$
of cubic forms as in~\cite[p.~2]{BD24} in each such point and deduce that the sheaf $\omega_{\Xf}^\vee$ is indeed globally generated.

Given a nonzero ideal $\qf$ of $\OK$, let $\pi_\qf \colon X(K)=\Xf(\OK) \to \Xf(\OK/\qf)$ be the reduction map, which factors through $\widetilde{\pi}_\qf \colon X(\Ab_K) \to \Xf(\OK/\qf)$. For $r \in \Xf(\OK/\qf)$, we want to estimate
\begin{equation*}
    N(H,r,B)\coloneq|\{x \in U(K) : H(x) \le B,\ \pi_\qf(x) = r\}|.
\end{equation*}

Let $\Delta_K$ be the discriminant of $K$, and $\rho_K$ be the residue of its Dedekind zeta function $\zeta_K(s)$ at $s=1$. In Section~\ref{sec:dp5}, all implicit constants in the error terms are allowed to depend on $K$, $H$, and $\qf$.

\begin{prop}\label{prop:dp5}
    Let $H$ be an admissible anticanonical height function, defined via a set $\Pc \subset \OK[Y_1,Y_2,Y_3]$ of cubic forms. Then
    \begin{equation*}
        N(H,r,B) = \alpha(X) \frac{\tau_H(X(\Ab_K))}{|\Xf(\OK/\qf)|}  B(\log B)^4(1+ O((\log\log B)^{-1/(3d+1)})).
    \end{equation*}
\end{prop}

More concretely, this modification of the Tamagawa measure `trivializes' the local densities for primes dividing $\qf$, in the sense that
\begin{equation}\label{eq:Tamagawa_volume_mod_q}
    \frac{\tau_H(X(\Ab_K))}{|\Xf(\OK/\qf)|} = \rho_K^5 |\Delta_K|^{-1}  \N\qf^{-2} \prod_{v\mid \qf} \lambda_v^{-1} 
    \prod_{v\nmid \qf}\lambda_v^{-1}\omega_{H,v}(X(K_v))
\end{equation}
with convergence factors
\begin{equation*}
    \lambda_\pf = \left(1-\frac{1}{\N\pf}\right)^{-5}
\end{equation*}
for finite primes $\pf$ and $\lambda_v=1$ for archimedean places $v$, and
\begin{equation*}
    \omega_{H,v}(X(K_v)) =
    \begin{cases}
        1+\frac{5}{\N\pf} + \frac{1}{\N\pf^2},
        & \text{if $v=\pf$ is finite}, \\
        \frac{3}{2} \vol\{y\in \Rd^3 : \max_{P\in \Pc}{|P(y)|_v\le 1}\}, & \text{if $v$ is real},\\
        \frac{12}{\pi} \vol\{y\in \Cd^3 : \max_{P\in \Pc}{|P(y)|_v\le 1}\}, & \text{if $v$ is complex.}
    \end{cases}
\end{equation*}
Indeed, all local factors for places not dividing $\qf$ can be found in~\cite[\S 1.2]{BD24}. If the anticanonical sections corresponding to $\Pc$ globally generate a suitable model, the induced Tamagawa measure is the model measure, assigning measure $\N\pf^{-2v_\pf(\qf)}$ to the residue disks $\{x\in \Xf(\of_\pf) : \congr{x}{r}{\pf^{v_p(\qf)}}\}$ as the model is smooth, whence $\N\qf^{-2}$ to their product.

   Note that, given $\qf$, the number of choices for $r \in \Xf(\OK/\qf)$ is 
    \begin{equation}\label{eq:dP5_mod_q}
        |\Xf(\OK/\qf)| = \N\qf^2 \prod_{\pf \mid \qf} \left(1+\frac{5}{\N\pf}+\frac{1}{\N\pf^2}\right)
    \end{equation}
    since $\Xf(\OK/\pf)$ has $\N\pf^2+5\N\pf+1$ points (because $\Xf$ is a blow-up of $\Pd_{\OK}^2$ in four points), since each such point has $\N\pf^{2(s-1)}$ lifts to $\Xf(\OK/\pf^s)$ by $s-1$ applications of Hensel's lemma, and using the Chinese remainder theorem to generalize to arbitrary $\qf$. Of course multiplying the leading constant in Proposition~\ref{prop:dp5} by this number of choices for $r$ recovers the leading constant in the main result of \cite{BD24}.

The remaining factor of the asymptotic formula is $\alpha(X)=1/144$~\citelist{\cite{Bre02}*{\S\,1.3} \cite{D07}*{Thm.~4}}.

For the proof, we incorporate the congruence condition into the proof of Manin's conjecture for $X$ in \cite{BD24}. In order to be reasonably brief, we have to assume detailed knowledge of that article, and we rely on its notation and conventions as in \cite[\S 1.5]{BD24}.
\subsection{Passage to a torsor}

We recall the five Pl\"ucker equations (see \cite{Skor93}, \cite[(3.6)]{BD24}) defining the universal torsor of $X$; here, we regard them as polynomials in ten variables with $\Zd$-coefficients:
\begin{equation*}
    \begin{aligned}
        f_1&=X_4X_{14}-X_3X_{13}+X_2X_{12},\\
        f_2&=X_4X_{24}-X_3X_{23}+X_1X_{12},\\
        f_3&=X_4X_{34}-X_2X_{23}+X_1X_{13},\\
        f_4&=X_3X_{34}-X_2X_{24}+X_1X_{14},\\
        f_0&=X_{12}X_{34}-X_{13}X_{24}+X_{23}X_{14}.
    \end{aligned}  
\end{equation*}
Using the grading of the Cox ring of $X$ by $\Pic(X)\cong \Zd^5$ with basis $\ell_0,\dots,\ell_4$, we have $\deg(X_i)=\ell_i$, $\deg(X_{jk})=\ell_0-\ell_j-\ell_k$ (see \cite[\S 3.1]{BD24}, for example), and hence
\begin{equation*}
    \deg(f_0)=2\ell_0-(\ell_1+\dots+\ell_4),\quad \deg(f_n)=\ell_0-\ell_n\text{ for }n=1,\dots,4.
\end{equation*}
Let $\Cc$ be a set of 5-tuples $\cfrb = (\cf_0,\dots,\cf_4)$ of fractional ideals representing $\Cl_K^5$.  Let $\Os_i = \cfrb^{\deg(X_i)}$, $\Os_{jk} = \cfrb^{\deg(X_{jk})}$, and $\ffr_n = \cfrb^{\deg(f_n)}$ for $n=0,\dots,4$ be the corresponding fractional ideals.

Recall the twisted universal torsors ${}_\cfrb\Yf \to \Xf$ from \cite[\S 2.2]{BD25integral}.
For each $\cfrb \in \Cc$, we consider the commutative diagram
\begin{equation}\label{eq:diagram}
   \begin{tikzcd}
 {}_\cfrb\Yf(\OK)\arrow[r, "\pi_\qf"] \arrow[d, "{}_\cfrb\rho"'] & 
 {}_\cfrb\Yf(\OK/\qf) \arrow[d, "{}_\cfrb\rho"] \\
X(K)=\Xf(\OK) \arrow[r, "\pi_\qf"']                    & \Xf(\OK/\qf)
    \end{tikzcd}
\end{equation}
and describe the points modulo $\qf$ in the upper right corner explicitly:

\begin{lemma}
    We have
    \begin{equation}\label{eq:cY(OK/q)}
        {}_\cfrb\Yf(\OK/\qf) =
        \left\{\begin{aligned}
            &\rb = (r_i+\qf \Os_i,r_{jk}+\qf\Os_{jk}) \in \Os_1/\qf\Os_1\times \dots\times\Os_{34}/\qf\Os_{34} :\\
            &\congr{f_n(\rb)}{0}{\ffr_n\qf}\text{ for }n=0,\dots,4,\\
            &\rfr_i+\rfr_j+\qf=\rfr_i+\rfr_{jk}+\qf=\rfr_{ij}+\rfr_{ik}+\qf=\OK 
        \end{aligned}\right\},
    \end{equation}
    where $\rfr_i=r_i\Os_i^{-1}$, $\rfr_{jk}=r_{jk}\Os_{jk}^{-1}$.
\end{lemma}

Here and in the rest of this section, we adopt the convention from \cite[\S 1.5]{BD24} regarding the indices $i,j,k,l$, which are meant to run through all pairwise distinct values in $\{1,2,3,4\}$, with $r_{jk}=r_{kj}$, for example.

\begin{proof}
The twisted torsor $\cYf$ is $\Spec(\cR)\setminus \Spec (\cI)$, where
\begin{equation*}
    \cR = \OK [\Oc_1^{-1}X_1,\dots,\Oc_{34}^{-1}X_{34}]/\sum_{n=0}^4\mathfrak{f}_n^{-1} f_n.
\end{equation*}
The irrelevant ideal $\cI$ is the product ideal in $\cR$ of all ideals of the form $X_i\Oc_i^{-1}+X_j\Oc_j^{-1}$, $X_i\Oc_i^{-1}+X_{jk}\Oc_{jk}^{-1}$, and $X_{ij}\Oc_{ij}^{-1}+X_{jk}\Oc_{jk}^{-1}$.

To start with, note that $\Hom(\Oc_i^{-1},\OK/\qf) = \Hom(\OK, \Oc_i/\qf\Oc_i)=\Oc_i/\qf\Oc_i$.
A homomorphism $\phi\colon \cR \to \OK/\qf$ is thus given by a tuple
\begin{equation*}
    (r_1+\qf\Oc_1,\dots,r_{34}+\qf\Oc_{34}) \in \Oc_1/\qf\Oc_1\times \cdots \times \Oc_{34}/\qf\Oc_{34}.
\end{equation*}
It factors through the quotient --- that is, defines an integral point on $\Spec(\cR)$ --- if $\mathfrak{f}_n^{-1} f_n\subset \ker(\phi)$. This in turn is equivalent to
\begin{equation*}
f_n \in \ker(\phi)\otimes_{\OK}\mathfrak{f}_n = \ker(\phi\otimes_{\OK} \mathfrak{f}_n),
\end{equation*}
given that fractional ideals are projective modules, yielding the second row on the right-hand side of~\eqref{eq:cY(OK/q)}.

Such an integral point on $\Spec(\cR)$ avoids $V(\cI)$ if and only if $\phi(X_i\Oc_i^{-1}+X_j\Oc_j^{-1})=\OK/\qf$ (and analogously for the other generators). This simply means $(r_i\Oc_i^{-1} +\qf)+ (r_j\Oc_j^{-1}+\qf) = \OK/\qf$, which in turn is equivalent to the last row on the right-hand side of~\eqref{eq:cY(OK/q)}.
\end{proof}

In the following, whenever we consider an element $\rb \in {}_\cfrb\Yf(\OK/\qf)$, we write $(r_i,r_{jk}) \in \Os_1 \times \dots \times \Os_{34}$ for a 10-tuple of representatives of its coordinates.

Recall that we must count rational points congruent to $r \in \Xf(\OK/\qf)$. In view of the diagram \eqref{eq:diagram}, this lifts to congruence conditions
\begin{equation}\label{eq:congruence_on_torsor}
    \pi_\qf(a_i,a_{jk}) \in {}_\cfrb\rho^{-1}(r)
\end{equation}
on the torsor ${}_\cfrb\Yf(\OK/\qf)$ as follows:

\begin{lemma}\label{lem:parameterization_dP5}
    We have a bijection
    \begin{equation*}
        \bigsqcup_{\cfrb \in \Cc} \Ms_{\cfb,r}(B) \to \{x \in U(K) : H(x) \le B,\ \pi_\qf(x) = r\},
    \end{equation*}
    where $\Ms_{\cfb,r}(B)$ is the set of $(\OK^\times)^5$-orbits on
    \begin{equation*}
        \Sc_{\cfrb,r}(B) \coloneq \{(a_i,a_{jk}) \in \Os_1^{\ne 0}\times \dots \times\Os_{34}^{\ne 0} : \eqref{eq:congruence_on_torsor},\ \text{\cite[(3.5)--(3.7)]{BD24}}\} \subset {}_\cfrb\Yf(\OK).
    \end{equation*}
    In other words, $(a_i,a_{jk}) \in {}_\cfrb\Yf(\OK)$ whose projection has height $\le B$ lies in $\Sc_{\cfrb,r}(B)$ if and only if there is an $\rb \in {}_\cfrb\Yf(\OK/\qf)$ above $r$ such that $\congr{a_i}{r_i}{\qf\Os_i}$, and similarly for $a_{jk}$.
\end{lemma}

\begin{proof}
    Recall \cite[Proposition~3.1]{BD24}, which parameterizes the points of bounded height in $U(K)$ in the same way, except that the congruence condition \eqref{eq:congruence_on_torsor} does not appear. The commutative diagram \eqref{eq:diagram} shows that $\pi_\qf(x)=r$ if and only if the corresponding points on the twisted torsors satisfy \eqref{eq:congruence_on_torsor}.
\end{proof}

\subsection{Symmetry}\label{sec:symmetry}

Recall the notation from \cite[\S 3.2]{BD24}. We define $\Msover_{\cfrb,r}^{(s)}(B)$, $\Msunder_{\cfrb,r}^{(s)}(B)$ by the symmetry conditions \cite[(3.11), (3.12)]{BD24}, respectively, as in \cite[Lemma~3.2]{BD24} with the additional congruence condition \eqref{eq:congruence_on_torsor}. By the same proof,
\begin{equation*}
    \sum_{s \in S} |\Msunder_{s(\cfrb),s(r)}^{(s)}(B)| \le |\Ms_{\cfrb,r}(B)| \le \sum_{s \in S} |\Msover_{s(\cfrb),s(r)}^{(s)}(B)|,
\end{equation*}
where $s(r)$ is obtained from $r$ by permutation of the coordinates and a sign change as described in the second paragraph of \cite[Lemma~3.2]{BD24}.

\subsection{Fundamental domain}

As in \cite[\S\S 3.3--3.4]{BD24}, we work with the sets 
\begin{equation*}
    \Os_*' = \Os_1^{\ne 0} \times \dots \times \Os_4^{\ne 0},\quad \Os'' = \Os_{12}\times \dots \times \Os_{34},
\end{equation*}
and the same fundamental domains $\Fs$, $\Fs_1$, $\Fs_0(\ab')$. We define $\Mover_{\cfrb,r}^{(s)}(B)$ as $\Mover_\cfrb^{(s)}(B)$ in \cite[\S 3.4]{BD24} with the additional congruence condition \eqref{eq:congruence_on_torsor}, satisfying the analogue
\begin{equation*}
     |\Msover_{\cfrb,r}^{(s)}(B)| = \frac{1}{|\mu_K|}\cdot |\Mover_{\cfrb,r}^{(s)}(B)|.
\end{equation*}
of \cite[(3.16)]{BD24}. We define $\Munder_{\cfrb,r}^{(s)}(B)$ analogously. As discussed in the third paragraph of \cite[\S 3.4]{BD24}, it suffices to consider $|\Mover_{\cfrb,r}^{(\id)}(B)|$.

\subsection{Restrictions}

The following result is the analogue of \cite[Proposition~4.7]{BD24}, which summarizes the restrictions of the counting problem in \cite[\S 4]{BD24}.

\begin{prop}
    We have
    \begin{equation*}
        |\Mover_{\cfrb,r}^{(\id)}(B)| = \sum_{\rb \in {}_\cfrb\rho^{-1}(r)} \sums{\ab' \in \Os_*' \cap \Fs_1^4\\\text{\cite[(4.9)]{BD24}}\\\congr{a_i}{r_i}{\qf\Os_i}} |A(W,\ab',\rb,\cfrb,B)| + O\left(\frac{B(\log B)^4}{(\log \log B)^{\frac{1}{3d+1}}}\right)
    \end{equation*}
    where $A(W,\ab',\rb,\cfrb,B)$ is defined as in \cite[\S 4.1]{BD24} with the six congruence conditions $\congr{a_{jk}}{r_{jk}}{\qf\Os_{jk}}$.
\end{prop}

For the proof, we note that the arguments in \cite[\S 4]{BD24} estimate error terms arising from restrictions of our counting problem. Due to our additional congruence conditions, we consider subsets of those sets, which only lead to smaller error terms.

\subsection{Dependent coordinates}

We generalize \cite[Lemma~3.3]{BD24} as follows.

\begin{lemma}\label{lem:dependent_aij}
    Let $\rb \in {}_\cfrb\rho^{-1}(r)$.
    Let $\ab' \in \Os_*'$ with $\congr{a_i}{r_i}{\qf\Os_i}$, and $(a_{12},a_{23},a_{34}) \in \Os_{12}\times\Os_{23}\times\Os_{34}$ with $\congr{a_{jk}}{r_{jk}}{\qf\Os_{jk}}$.
    If 
    \begin{equation}\label{eq:mod_a4}
        \congr{a_3a_{23}}{a_1a_{12}+a_4r_{24}}{a_4\Os_{24}\qf}
    \end{equation}
    and
    \begin{equation}\label{eq:mod_a1}
        \congr{a_4a_{34}}{a_2a_{23}+a_1r_{13}}{a_1\Os_{13}\qf}
    \end{equation}
    hold, then we obtain unique
    \begin{equation*}
        \begin{aligned}
            a_{13}&=\frac{a_2a_{23}-a_4a_{34}}{a_1},\\
            a_{24}&=\frac{a_3a_{23}-a_1a_{12}}{a_4},\\
            a_{14}&=\frac{a_2a_3a_{23}-a_3a_4a_{34}-a_1a_2a_{12}}{a_1a_4}
        \end{aligned}
    \end{equation*}
    satisfying the torsor equations \cite[(3.6)]{BD24}, with $\congr{a_{13}}{r_{13}}{\qf\Os_{13}}$ and $\congr{a_{24}}{r_{24}}{\qf\Os_{24}}$. If additionally $\af_1+\af_4=\OK$, then $\congr{a_{14}}{r_{14}}{\qf\Os_{14}}$. If \eqref{eq:mod_a4} or \eqref{eq:mod_a1} does not hold, \cite[(3.6)]{BD24} has no solution satisfying the congruence conditions $\congr{a_{jk}}{r_{jk}}{\qf\Os_{jk}}$.
\end{lemma}

\subsection{M\"obius inversions}

We will encounter the conditions
\begin{align}
    &\dfr_i+\af_j=\OK,\quad \dfr_i+\qf = \OK,\label{eq:dfrb}\\
    &\efr_i \mid \af_i,\quad \efr_i+\qf = \OK,\label{eq:efrb}\\
    &\ffr_{ij}\coloneq(\dfr_i\cap\dfr_j)\efr_k\efr_l \mid \af_{ij},\label{eq:moebius}
\end{align}
with $\dfrb = (\dfr_1,\dots,\dfr_4), \efrb = (\efr_1,\dots,\efr_4) \in \IK^4$.

Let $\Gs = \Gs(\cfrb,\ab',\rb,\dfrb,\efrb)$ be the subset of $K^3$ consisting of all $(a_{12},a_{23},a_{34})$ for which all six coordinates $a_{ij}$ (obtained via Lemma~\ref{lem:dependent_aij}) are integral and satisfy $\congr{a_{ij}}{r_{ij}}{\qf\Os_{ij}}$ and \eqref{eq:moebius}.

\begin{prop}
    We have
    \begin{align*}
        |\Mover_{\cfrb,r}^{(\id)}(B)| ={}&\sum_{\rb \in {}_\cfrb\rho^{-1}(r)} \sums{\ab' \in \Os_*' \cap \Fs_1^4\\\text{\cite[(4.9)]{BD24}}\\\congr{a_i}{r_i}{\qf\Os_i}} \theta_0(\afrb') 
        \sums{\dfrb: \eqref{eq:dfrb},\ \N\dfr_i\le T_1\\\efrb:\eqref{eq:efrb},\ \N\efr_i \le T_1}\mu_K(\dfrb,\efrb)\\
        &\times |\Gs(\cfrb,\ab',\rb,\dfrb,\efrb)\cap \Fs_0^{(W)}(\ab';u_\cfrb B)|
        + O\left(\frac{B(\log B)^4}{(\log \log B)^{\frac{1}{3d+1}}}\right).
    \end{align*}
\end{prop}

\begin{proof}
    We apply M\"obius inversion as in \cite[Lemma~5.1]{BD24}. Here, we may restrict the M\"obius variables $\dfr_i,\efr_i$ to be coprime to $\qf$ since $\Gs$ is empty otherwise because of the congruence conditions modulo $\qf$.

    Next, we restrict the range of $\dfr_i,\efr_i$ as in \cite[Proposition~5.4]{BD24}; since we consider the subset satisfying our congruence conditions, the error terms can only be smaller.
\end{proof}

We define the fractional ideals
\begin{align*}
    \bfr_{12} &\coloneq (\dfr_1\cap\dfr_2\cap(\dfr_3+\dfr_4))\efr_3\efr_4\Os_{12}\qf,\\
    \bfr_{23} &\coloneq (\dfr_2\cap\dfr_3\cap\dfr_4)\efr_1\Os_{23}\qf,\\
    \bfr_{34} &\coloneq (\dfr_1\cap\dfr_3\cap\dfr_4)\efr_2\Os_{34}\qf.
\end{align*}
Then $\Gs = \gammab + \Gs'$ for some $\gammab = \gammab(\ab',\rb,\cfrb,\dfrb,\efrb) \in K^3$ and an additive subgroup $\Gs' = \Gs'(\ab',\cfrb,\dfrb,\efrb)$ of $K^3$ given by congruence conditions on $a_{12},a_{23},a_{34}$ modulo $\bfr_{12},\bfr_{23},\bfr_{34}$, respectively.

As in \cite[Lemma~5.5, Proposition~5.7]{BD24}, we obtain
\begin{align*}
    |\Gs(\cfrb,\ab',\rb,\dfrb,\efrb)\cap \Fs_0^{(W)}(\ab';u_\cfrb B)| ={}& \frac{2^{3r_2}\vol S_F^{(W)}(\ab';u_\cfrb B)}{|\Delta_K|^{\frac 3 2}\N(\af_1\af_4\bfr_{12}\bfr_{23}\bfr_{34})}\\
    &+O\left(\frac{W^{3d}B}{T_2^{1/3}|N(a_1\cdots a_4)|}\right),
\end{align*}
where the total contribution of the error term is negligible.

By \cite[Lemma~5.8]{BD24} and \cite[Lemma~5.12]{BD24} (where our restriction to $\dfr_i,\efr_i$ coprime to $\qf$ changes the Euler factors of $\theta(\afrb')$ to $1$ for all $\pf \mid \qf$), we obtain the following analogue of \cite[Proposition~5.13]{BD24} (where the sum over $a_1,\dots,a_4$ and $\cfrb$ is not turned into a sum over ideals $\af_1,\dots,\af_4$ because of our congruence conditions on $a_i$):

\begin{prop}\label{prop:result_of_sum_over_aij}
    We have
    \begin{align*}
        |\Mover_{\cfrb,r}^{(\id)}(B)| ={}& \frac{2^{3r_2}\vol(S_F^{(W)}(\oneb;1))}{|\Delta_K|^{3/2}} \sum_{\rb \in {}_\cfrb\rho^{-1}(r)} \sums{\ab' \in \Os_*' \cap \Fs_1^4\\\text{\cite[(4.9)]{BD24}}\\\congr{a_i}{r_i}{\qf\Os_i}} \frac{\theta_0(\afrb')\theta(\afrb')B}{\N(\af_1\af_2\af_3\af_4)}\\
        &+ O\left(\frac{B(\log B)^4}{(\log \log B)^{\frac{1}{3d+1}}}\right),
    \end{align*}
    where
    \begin{equation*}
        \theta(\afrb') = \frac{1}{\N\qf^3}\prods{\pf \mid \af_1\af_2\af_3\af_4\\\pf \nmid \qf} \left(1-\frac{1}{\N\pf}\right)\left(1-\frac{1}{\N\pf^2}\right) \cdot \prods{\pf \nmid \af_1\af_2\af_3\af_4\\\pf \nmid \qf} \left(1-\frac{4}{\N\pf^2}+\frac{3}{\N\pf^3}\right).
    \end{equation*}
\end{prop}

\subsection{Completion of the proof}

For the summations over $a_1,\dots,a_4$, we need the following generalization of \cite[Corollary~2.7]{DF14} (similar to \cite[Corollary~6.9]{D09}, which allows a congruence condition, but only for congruence classes coprime to the modulus):

\begin{lemma}\label{lem:sum_with_congruence}
    Let $\theta \in \Theta(\bfr,C_1,C_2,C_3)$ (see \cite[Definition~2.1]{DF14}) with corresponding local factors $\theta_\pf(n)$ (i.e., $\theta(\af) = \prod_\pf \theta_\pf(v_\pf(\af))$ for all $\af \in \IK$). Let $\qf$ be an ideal in $\OK$, and $\Os$ be a fractional ideal of $K$. Let $r \in \Os$, and $\rfr = r\Os^{-1}$. Then
    \begin{align*}
        \sum_{\lambda+\qf \in (\OK/\qf)^\times}\!\sums{a \in \Os^{\ne 0}\cap \Fs_1\\\N(a\Os^{-1}) \le t\\\congr{a}{\lambda r}{\qf\Os}} \!\theta(a\Os^{-1}) = \frac{\rho_K}{h_K} \theta' t+O_{C_2}(\tau_K(\bfr\qf)(C_1'C_2)^{\omega_K(\bfr\qf)}C_3 t^{1-1/d}),
    \end{align*}
    where
    \begin{equation*}
        \theta' = \prod_\pf \theta'_\pf
    \end{equation*}
    with
    \begin{equation*}
        \theta'_\pf = \begin{cases}
            \left(1-\frac{1}{\N\pf}\right)\sum_{n=0}^\infty\frac{\theta_{\pf}(n)}{\N\pf^n}, & \pf \nmid \qf,\\
            \left(1-\frac{1}{\N\pf}\right)\theta_{\pf}(v_\pf(\rfr)), & v_\pf(\rfr) < v_\pf(\qf),\\
            \left(1-\frac{1}{\N\pf}\right)^2\sum_{n=v_\pf(\qf)}^\infty\frac{\theta_{\pf}(n)}{\N\pf^{n-v_\pf(\qf)}}, &v_\pf(\rfr) \ge v_\pf(\qf)>0,
        \end{cases}
    \end{equation*}
    and $C_1'=\max_{\pf\mid\qf}\theta_\pf(0)+\max\{C_1,C_2\}\max_{\pf\mid\qf}v_\pf(\qf)$.
\end{lemma}

\begin{proof}
    We observe that $\congr{\lambda r}{\lambda' r}{\qf\Os}$ if and only if $\congr{\lambda}{\lambda'}{\qf'}$ where $\qf' = \qf(\rfr+\qf)^{-1}$. Therefore, the number of $\lambda+\qf$ giving the same sum over $a$ is
    \begin{equation*}
        \frac{\phi_K(\qf)}{\phi_K(\qf')}=\N(\rfr+\qf)\prod_{\pf : v_\pf(\rfr) \ge v_\pf(\qf)} \left(1-\frac{1}{\N\pf}\right).
    \end{equation*}

    We claim that $a \in \Os$ satisfies $\congr{a}{\lambda r}{\qf\Os}$ for some $\lambda$ that is coprime to $\qf$ if and only if $a\Os^{-1}+\qf = \rfr+\qf$. Indeed,  the existence of such a $\lambda$ immediately implies both inclusions $r \in (a)+\qf\Os$ and $a \in (r)+\qf\Os$ of the claimed identity, while conversely assuming the identity, we obtain the existence of a suitable $\lambda$ first locally for each prime power dividing $\qf\Os$ and then modulo $\qf\Os$ by the Chinese remainder theorem. 
    
    Therefore, denoting the ideal class of $\Os$ by $\kf$, our sums over $\lambda,a$ are equal to
    \begin{equation}\label{eq:sum_over_ideals}
        \frac{\phi_K(\qf)}{\phi_K(\qf')} \sums{\af \in \IK\cap\kf\\\N\af \le t\\\af+\qf = \rfr+\qf} \theta(\af).
    \end{equation}
    
    We observe that $\af+\qf = \rfr+\qf$ if and only if $v_\pf(\af) = v_\pf(\rfr)$ for all $\pf$ with $v_\pf(\rfr)<v_\pf(\qf)$ and $v_\pf(\af) \ge v_\pf(\qf)$ for all $\pf$ with $v_\pf(\rfr) \ge v_\pf(\qf)$. Therefore, we can leave out the condition $\af+\qf = \rfr+\qf$ if we replace the local factors $\theta_\pf(n)$ of $\theta$ by $0$ if $n \ne v_\pf(\rfr) < v_\pf(\qf)$ or $n < v_\pf(\qf) \le v_\pf(\rfr)$. Since the original $\theta$ is in $\Theta(\bfr,C_1,C_2,C_3)$, we can check that our redefined $\theta$ is in $\Theta(\bfr\qf,C_1',C_2,C_3)$. Therefore, we can apply \cite[Proposition~2.3]{DF14} (whose definition of $\rho_K$ differs from our convention by a factor of $h_K$).
\end{proof}

\begin{remark}\label{rem:sum_with_congruence_simplified}
    Often Lemma~\ref{lem:sum_with_congruence} simplifies as follows:
    \begin{itemize}
        \item If the local factors of $\theta$ satisfy $\theta_\pf(n) \le 1$ for all $n$ and $\pf$, then $\theta_\pf' \le 1$ for all $\pf$. In this case, $C_1' \ll_{C_1,C_2,\qf} 1$.
        \item If $\theta_\pf(n)=\theta_\pf(1)$ for all $n \ge 1$, we obtain
        \begin{equation*}
            \theta'_\pf = \begin{cases}
                \left(1-\frac{1}{\N\pf}\right)\theta_\pf(0)+\frac{1}{\N\pf}\theta_\pf(1),& \pf \nmid \qf,\\
                \left(1-\frac{1}{\N\pf}\right)\theta_\pf(0), & \pf \mid \qf,\ \pf \nmid \rfr,\\
                \left(1-\frac{1}{\N\pf}\right)\theta_\pf(1), & \pf \mid \qf,\ \pf \mid \rfr.
            \end{cases}
        \end{equation*}
    \end{itemize}
\end{remark}

\begin{proof}[Proof of Proposition~\ref{prop:dp5}]
For the summation over $\rb \in {}_\cfrb\rho^{-1}(r)$ in Proposition~\ref{prop:result_of_sum_over_aij}, we can choose one element $\rb$ and sum over $(\lambda_0,\dots,\lambda_4)$ running through representatives of $\Gm(\OK/\qf)^5$ and acting on $\rb$ by mapping a representative $(r_i,r_{jk})$ to $(\lambda_i r_i,\lambda_0\lambda_j^{-1}\lambda_k^{-1} r_{jk})$. Since the terms in the sum over $\rb$ are independent of $r_{jk}$, the sum over $\lambda_0 \in (\OK/\qf)^\times$ simply gives a factor $\phi_K(\qf)$. We define $\rfr_i = r_i\Os_i^{-1}$ since we must sum over $a_i \in \Os_i$ congruent to $\lambda_i r_i$ modulo $\qf\Os_i$.

Lemma~\ref{lem:sum_with_congruence} and Remark~\ref{rem:sum_with_congruence_simplified} can be applied to a summation of $\N\qf^3\theta_0(\afrb')\theta(\afrb')$ (which lies in $\Theta(\prod_{\pf \mid \af_1\af_2\af_3,\pf \nmid \qf} \pf,1,1,1)$ as a function in $\af_4$) over $\lambda_4$ and $a_4$. Here, the resulting $\theta'$ depends on $\af_1,\af_2,\af_3,\rfr_4,\qf$ and can be computed to be $\theta' = \prod_\pf \theta'_\pf$ with
\begin{equation*}
    \theta'_\pf = \begin{cases}
        \left(1-\frac{1}{\N\pf}\right)^2\left(1+\frac{2}{\N\pf}-\frac{2}{\N\pf^2}\right), &\pf \nmid \af_1\af_2\af_3\qf,\\
        \left(1-\frac{1}{\N\pf}\right)^3\left(1+\frac{1}{\N\pf}\right), &\pf \mid \af_i,\ \pf \nmid \af_j\af_k\qf,\\
        \left(1-\frac{1}{\N\pf}\right), &\pf \mid \qf,\ \text{$\pf$ divides at most one of $\af_1,\af_2,\af_3,\rfr_4$},\\
        0, &\text{else.}
    \end{cases}
\end{equation*}
Since this sum over $\lambda_4,a_4$ can be viewed as a summation over ideals in a certain class as discussed after \eqref{eq:sum_over_ideals}, this is compatible with partial summation as in \cite[Lemma~2.10]{DF14}. Therefore, we can argue as in the application of \cite[Proposition~7.2]{DF14} in the proof of \cite[Lemma~5.14]{BD24} to perform the summation over $\lambda_4,a_4$ in the main term of Proposition~\ref{prop:result_of_sum_over_aij}.

We observe that $\theta'$ as a function in $\af_3$ lies in $\Theta(\prod_{\pf \mid \af_1\af_2,\pf \nmid \qf} \pf,1,2,1)$.
Hence
we can repeat these steps for $\lambda_3,a_3$, resulting in $\theta'' = \prod_\pf \theta''_\pf$ (depending on $\af_1,\af_2,\rfr_3,\rfr_4,\qf$) with
\begin{equation*}
    \theta''_\pf = \begin{cases}
        \left(1-\frac{1}{\N\pf}\right)^3\left(1+\frac{3}{\N\pf}-\frac{1}{\N\pf^2}\right), &\pf \nmid \af_1\af_2\qf,\\
        \left(1-\frac{1}{\N\pf}\right)^4\left(1+\frac{1}{\N\pf}\right), &\pf \nmid \qf,\ \text{$\pf$ divides at most one of $\af_1,\af_2$},\\
        \left(1-\frac{1}{\N\pf}\right)^2, &\pf \mid \qf,\ \text{$\pf$ divides at most one of $\af_1,\af_2,\rfr_3,\rfr_4$},\\
        0, &\text{else.}
    \end{cases}
\end{equation*}
As a function in $\af_2$, we have $\theta'' \in \Theta(\prod_{\pf \mid \af_1,\pf \nmid \qf} \pf,1,3,1)$.

Next, summing over $\lambda_2,a_2$ yields $\theta''' = \prod_\pf \theta'''_\pf$ with
\begin{equation*}
    \theta'''_\pf = \begin{cases}
        \left(1-\frac{1}{\N\pf}\right)^4\left(1+\frac{4}{\N\pf}\right), &\pf \nmid \af_1\qf,\\
        \left(1-\frac{1}{\N\pf}\right)^5\left(1+\frac{1}{\N\pf}\right), &\pf \mid \af_1,\ \pf \nmid \qf,\\
        \left(1-\frac{1}{\N\pf}\right)^3, &\pf \mid \qf,\ \text{$\pf$ divides at most one of $\af_1,\rfr_2,\rfr_3,\rfr_4$},\\
        0, &\text{else.}
    \end{cases}
\end{equation*}
As a function in $\af_1$, we have $\theta''' \in \Theta(\OK,1,4,1)$.

Finally, since \eqref{eq:cY(OK/q)} ensures that every $\pf \mid \qf$ divides at most one of $\rfr_1,\dots,\rfr_4$, summing over $\lambda_1,a_1$ gives
\begin{equation*}
    \theta'''' = \prod_{\pf \nmid \qf}\left(1-\frac{1}{\N\pf}\right)^5\left(1+\frac{5}{\N\pf}+\frac{1}{\N\pf^2}\right)\prod_{\pf \mid \qf} \left(1-\frac{1}{\N\pf}\right)^4.
\end{equation*}

In total, we obtain
\begin{equation*}
    \sum_{\rb \in {}_\cfrb\rho^{-1}(r)}\sums{\ab' \in \Os_*' \cap \Fs_1^4\\\text{\cite[(4.9)]{BD24}}\\\congr{a_i}{r_i}{\qf\Os_i}} \frac{\theta_0(\afrb')\theta(\afrb')B}{\N(\af_1\af_2\af_3\af_4)} = \frac{3 \alpha(X)\rho_K^4}{5h_K^4} \theta_1 B(\log B)^4 + O\left(\frac{B(\log B)^4}{\log \log B}\right),
\end{equation*}
where
\begin{equation*}
    \theta_1 = \frac{\phi_K(\qf)}{\N\qf^3}\prod_{\pf\nmid \qf}\left(1-\frac{1}{\N\pf}\right)^5\left(1+\frac{5}{\N\pf}+\frac{1}{\N\pf^2}\right)\prod_{\pf \mid \qf} \left(1-\frac{1}{\N\pf}\right)^4,
\end{equation*}
which may be rewritten as 
\begin{equation*}
        \theta_1 = \frac{1}{|\Xf(\OK/\qf)|}\prod_{\pf}\left(1-\frac{1}{\N\pf}\right)^5\left(1+\frac{5}{\N\pf}+\frac{1}{\N\pf^2}\right)
\end{equation*}
using~\eqref{eq:dP5_mod_q} and $\phi_K(\qf)/\N\qf =\prod_{\pf\mid\qf}(1-\N\pf^{-1})$.

As in the final part of \cite{BD24} (using \cite[Lemma~5.10, Lemma~5.11]{BD24}), we remove the dependence on the parameter $W$ and observe that the sum over $s \in S$ in Section~\ref{sec:symmetry} gives a factor $5$; furthermore, $|\Munder_{\cfrb,r}^{(\id)}(B)|$ can be estimated with the same main term. The result is independent of $\cfrb$, so that the summation over $\cfrb \in \Cc$ from Lemma~\ref{lem:parameterization_dP5} contributes a factor $h_K^5$.
This completes the proof of Proposition~\ref{prop:dp5}, with leading constant $\alpha(X)$ times the Tamagawa volume as in \eqref{eq:Tamagawa_volume_mod_q}.
\end{proof}

\begin{proof}[Proof of Theorem~\ref{thm:equidistribution_dP5}]
    This is now a direct consequence of applying the abstract equidistribution theorem (with the set $S$ of archimedean places and in the form of Corollary~\ref{cor:equidistribution-by-counting-mod-q}) to Proposition~\ref{prop:dp5}, using Remark~\ref{rmk:salberger-model-measure}. As the model $\Xf$ is moreover smooth, $c_r = |\Xf(\OK/\qf)|^{-1}$ for all residue disks modulo ideals $\qf$.
\end{proof}

\appendix
\section{A converse abstract equidistribution theorem}

This appendix records the converse to the abstract equidistribution theorem found in the work of Chambert-Loir and Tschinkel~\cite[Prop.~2.10]{CLT10}, in the same general setting. The proof is straightforward and analogous to Peyre's~\cite[Prop.~3.3~(b)]{Pey95}.

\begin{theorem}
    Let $X$ be a compact topological space and $H\colon X\to \Rd_{>0}$ continuous. Let $\nu$ be a measure such that for all $B>0$, the set $\{H\le B\}$ has finite $\nu$-measure.
    Assume that the sequence
    \begin{equation*}
        \df\nu_B =  \frac{\ind_{\{H\le B\}}}{\nu\{H\le B\}} \df \nu
    \end{equation*}
    of probability measures converges weakly to a probability measure $\mu$, as $B \to \infty$.

    Let $\theta\colon X\to \Rd_{>0}$ be a continuous function that is bounded from above and below. Then
    \begin{equation}\label{eq:equi-converse}
        \frac{\nu\{\theta H \le B\}}{\nu\{H\le B\}} \to \int_X \theta^{-1} \df \mu.
    \end{equation}
    In particular, the sequence
    \begin{equation*}
        \df\nu'_B =  \frac{\ind_{\{\theta H\le B\}}}{\nu\{\theta H\le B\}} \df \nu
    \end{equation*}
    of probability measures converges weakly to the probability measure
    \begin{equation*}
        \frac{\df \mu}{\theta\norm{\theta^{-1}}_1}.
    \end{equation*}
\end{theorem}

\begin{proof}
Let $\Uc = \{U\subset X : \text{$U$ is Borel, $\mu(\partial U)=0$} \}$. Note that $\Uc$ contains $\emptyset$ and $X$ and that it is stable under complements ($\partial U^c = \partial U$) and finite unions ($\partial (U \cup V)\subset \partial U \cup \partial V$), hence also under finite intersections and differences
(i.e., it is an algebra).
        
The statement~\eqref{eq:equi-converse} is clear for positive multiples of characteristic functions $\ind_Z$, $Z\in\Uc$, hence also for positive linear combinations of such functions for disjoint $Z$. Refining the sets $Z$, it becomes true for linear combinations for not necessarily disjoint $Z$.
Let $\theta_{\epsilon}$ be a family of such linear combinations converging uniformly to $\theta$. For instance, the set $\{\theta=\norm{\theta}_\infty/2+\delta\}$ must be a $\mu$-null set for all but countably many small $\delta$, so that one may start
with $\theta_1 = 2^{-1}\norm{\theta}_\infty\ind_{\theta\ge \norm{\theta}_\infty/2+\delta}$ satisfying $\norm{\theta-\theta_1}\le \norm{\theta}_\infty/2+\delta$ and proceed inductively.

Using this sequence, we have
\begin{align*}
    &\limsup_{B\to \infty}
    \left\lvert\frac{\nu\{\theta H\le B\}}{\nu\{H\le B\}}
    - \frac{\nu\{\theta_{\epsilon} H\le B\}}{\nu\{H\le B\}}\right\rvert \\
    & \qquad \le 
    \limsup_{B\to \infty}
    \frac{
        \nu\{(\theta_\epsilon+\epsilon)H\le B\} - \nu\{(\theta_\epsilon-\epsilon) H \le B\}
    }{
        \nu\{H\le B\}
    }
\end{align*}
Using~\eqref{eq:equi-converse} for $\theta_\epsilon+\epsilon$ and $\theta_\epsilon-\epsilon$ (as soon as $\epsilon\le \min_{x\in X}{\theta_\epsilon(x)}/2$),
this $\limsup$ is in fact a limit and equal to
\[
    \int  (\theta_\epsilon + \epsilon)^{-1}\df \mu -  \int (\theta_\epsilon-\epsilon)^{-1} \df \mu \ll_\theta \epsilon.
\]
Therefore,
\[
    \limsup_{B\to \infty}
    \left\lvert\frac{\nu\{\theta H\le B\}}{\nu\{H\le B\}}
    -\int \theta^{-1} \df \mu\right\rvert
    \le
    \limsup_{B\to \infty}
    \left\lvert\frac{\nu\{\theta_\epsilon H\le B\}}{\nu\{H\le B\}}
    -\int \theta^{-1}_\epsilon \df \mu\right\rvert
    + O_\theta(\epsilon).
\]
The second $\limsup$ is $0$ using~\eqref{eq:equi-converse} for $\theta_\epsilon$, finishing the proof of the first part of the theorem. The second part now follows from the abstract equidistribution theorem by Chambert-Loir and Tschinkel~\cite[Prop.~2.10]{CLT10}.
\end{proof}

\bibliographystyle{amsalpha}

\bibliography{dp5}

\end{document}